\documentclass{amsart}
\usepackage[foot]{amsaddr}

\title{Quantifying CDS Sortability of Permutations\\ by Strategic Pile Size}

\author{Marisa Gaetz$^1$}
\address{$^1$Massachusetts Institute of Technology}
\email{mgaetz@mit.edu}

\author{Bethany Flanagan$^2$}
 \address{$^2$Purdue University}
\email{bmolokac@purdue.edu}

\author{Marion Scheepers$^3$}
 \address{$^3$Boise State University}
\email{mscheepe@boisestate.edu}

\author{Meghan Shanks$^4$}
\address{$^4$University of Illinois at Urbana-Champaign}
\email{meghans2@illinois.edu}

\usepackage{amssymb, amsmath, amsthm, amsfonts} 
\usepackage{subcaption}
\usepackage{mathrsfs}
\usepackage{hyperref}
\usepackage[utf8]{inputenc}
\usepackage{caption}
\usepackage{graphicx}
\usepackage{tikz}
\usepackage{color,soul}
\usepackage{calc}
\usepackage[normalem]{ulem}
\usepackage{enumitem}
\usepackage[toc,page]{appendix}
\usetikzlibrary{arrows,positioning,automata,shapes,calc,3d,graphs}
\usetikzlibrary{decorations.markings}
\usetikzlibrary{arrows, shadows}
\usepackage[underline=true,rounded corners=false]{pgf-umlsd}
\usepackage{xspace}
\usepackage{verbatim}
\usepackage[margin=1in]{geometry}
\usepackage{xassoccnt}

\newtheorem{thm}{Theorem}[section]

\newtheorem{cor}[thm]{Corollary}
\newtheorem{lemma}[thm]{Lemma}
\newtheorem{remark}[thm]{Remark}
\newtheorem{prop}[thm]{Proposition}
\newtheorem{observation}[thm]{Observation}

\theoremstyle{definition}
\newtheorem{definition}[thm]{Definition}
\newtheorem{defn}[thm]{Definition}
\newtheorem{example}[thm]{Example}
\newtheorem{alg}[thm]{Algorithm}
\newtheorem{notation}[thm]{Notation}

\newcommand{\cds}{{\upshape{\textbf{CDS}}}\xspace} 

\DeclareMathOperator{\orb}{orb}
\DeclareMathOperator{\stab}{stab}

\theoremstyle{plain}

\newtheorem*{repp@theorem}{\repp@title (reformulated)}
\newcommand{\newrepptheorem}[2]{%
\newenvironment{repp#1}[1]{%
 \def\repp@title{#2 \ref{##1}}%
\begin{repp@theorem}}%
 {\end{repp@theorem}}}
\makeatother
\newrepptheorem{theorem}{Theorem}

\theoremstyle{definition}
\makeatletter
\newtheorem*{rep@theorem}{\rep@title \ continued}
\newcommand{\newreptheorem}[2]{%
\newenvironment{rep#1}[1]{%
 \def\rep@title{#2 \ref{##1}}%
 \begin{rep@theorem}}%
 {\end{rep@theorem}}}
\makeatother
\newreptheorem{example}{Example}

\newcommand{\vertex}{\node[vertex]}

\tikzstyle{vertex}=[circle, draw, inner sep=0pt, minimum size=10pt]
 \usepackage{graphicx, scalefnt}
 \usepackage{pgfplots}
 \usepackage{etex}
 \usetikzlibrary{decorations.pathmorphing}
\usetikzlibrary{shapes,arrows}
\tikzstyle{decision} = [diamond, draw, text width=5em, text badly centered, node distance=3cm, inner sep=0pt]
\tikzstyle{block} = [rectangle, draw, text width=5em, text centered, rounded corners, minimum height=4em]

\definecolor{lightgrey}{RGB}{216,227,225}

\begin{document}
\begin{abstract}
The special purpose sorting operation, \emph{context directed swap} (\cds), is an example of the block interchange sorting operation studied in prior work on permutation sorting. \cds has been postulated to model certain molecular sorting events that occur  in the  genome maintenance program of some species of ciliates. We investigate the mathematical structure of permutations not sortable by the \cds sorting operation. In particular, we present substantial progress towards quantifying permutations with a given \textit{strategic pile} size, which can be understood as a measure of \cds non-sortability. Our main results include formulas for the number of permutations in $\textsf{S}_n$ with maximum size strategic pile. More generally, we derive a formula for the number of permutations in $\textsf{S}_n$ with strategic pile size $k$, in addition to an algorithm for computing certain coefficients of this formula, which we call \textit{merge numbers}. 
\end{abstract}

\maketitle

\noindent \textit{2010 Mathematics Subject Classification. 05A05, 05A17}

\noindent \textit{Keywords: permutation sorting, context directed swap, strategic pile, factorization into cycles}

\section{Introduction}

Sorting is a fundamental step in numerous natural, industrial, commercial, and  scientific computing processes. Correspondingly, the mathematical analysis of sorting operations has a long history. The typical concerns with a sorting process include the efficiency of the sorting operation, a characterization of the situations in which the sorting operation achieves the sorting objective, and a characterization of the situations in which the sorting operation does not achieve the sorting objective. In this paper we focus on the third of these concerns. In particular, we seek to quantify for a specific sorting operation the prevalence of what can be seen as the worst case obstruction to sortability.

The specific sorting operation we consider aims to sort a permuted list of the numbers $1, 2, \ldots, n$ to the canonical ordered list $( 1, 2, \ldots,  n)$. This sorting operation appears in two prior works. It appears in the 2003 template model for the construction of a new macronucleus from its scrambled precursor micronucleus in certain ciliate species \cite{PER1}. In this model the sorting operation is named \emph{dlad}. For more on this fascinating biological background the reader may consult the review \cite{PrescottGymnastics} and the textbook \cite{CiliateText}. It turns out, by hindsight, that this sorting operation also includes special cases of the block interchange sorting operation examined in \cite{Christie} by Christie (1996). The \emph{minimal block interchanges} identified by Christie are special cases of the \emph{dlad} operation. 

In yet another investigation into genome rearrangement combinatorics, the double cut and join operation, denoted DCJ, is introduced by Yancopoulos \emph{et al.} \cite{DCJ1} (2005) to establish a mathematical measure of distance between two genomes. In the DCJ theory, generic block interchanges studied by Christie in \cite{Christie} are modeled by a very specific sequence of DCJ events, visualized in \cite[Figure 6]{DCJ1}. Modeling \emph{dlad} as a DCJ operation requires specifying additional DCJ constraints. To emphasize the specific mathematical nature of the sorting operation we consider here, the operation will be called \emph{context directed swap}, denoted \cds. We base our treatment on the mathematical counterpart of the essential features identified in the paper \cite{PER1}. 

A permuted list of numbers is said to be \cds-\emph{sortable} if there is a sequence of applications of the \cds sorting operation (to be defined in the next section) that results in the numbers listed in increasing order. Not every permutation is sortable by \cds. \cds-sortability criteria have been given previously (for instance, see \cite{cdssort}). Also, from prior work one can deduce that if a permutation is \cds-sortable, then sorting by applications of \cds provides the most efficient sorting by block interchanges. Mathematically interesting phenomena arise from the study of permutations not sortable by applications of \cds. The essential structural obstacle to a permutation's \cds-sortability was identified in \cite{AHMMSTW}, giving rise to the notion of the \emph{strategic pile} of a permutation. 

The notions of \cds-sortability, the strategic pile of a permutation, and appropriate notation and terminology will be introduced in Section 1 below. In this section we explicitly describe the problem being treated in this paper, and we report our findings in Sections 2 through 4.

In section 2, we determine the number of elements in $\textsf{S}_n$ that have the maximum size strategic pile among all elements of $\textsf{S}_n$. 
This counting problem reduces to a variation of the cycle factoring problem for $\textsf{S}_n$, studied previously, and on the cycle factoring results of \cite{BertramWei} and \cite{cangelmi}. In section 3, we investigate how prevalent it is for  permutations in $\textsf{S}_n$ to have strategic piles of cardinality $k$. As a result of this work we develop formulas in closed form that produce the terms of the integer sequences A267323, A267324 and A267391 in \cite{oeis}. We also contribute the integer sequence A281259 to \cite{oeis}, as well as its formula. In section 4, we highlight a more challenging component of our formula from section 3.

\section{Preliminaries}

For a positive integer $n$, the symbol $\textsf{S}_n$ denotes the set of one-to-one functions from the set $\{1, 2, \ldots, n\}$ to itself, also known as permutations of $\{1, 2, \ldots, n\}$.
The notation
\begin{equation}\label{eq:genperm}
  \lbrack a_1\; a_2\; \cdots\; a_{n-1}\; a_n\rbrack
\end{equation}
denotes the permutation $\pi$ for which $\pi(i)=a_i$ for $1\le i\le n$. 
In current literature the notation in (\ref{eq:genperm}) is called \emph{one line notation}. This one-line notation
should be distinguished from
\begin{equation}\label{eq:cycle}
 (c_1\; c_2\; \cdots\; c_{k-1}\; c_k),
\end{equation}
which is the so-called \emph{cycle notation} that denotes the permutation $\pi$ where $\pi(c_1) = c_2,\; \pi(c_2) = c_3, \ldots ,\; \allowbreak \pi(c_{k-1}) = c_k,\; \pi(c_k) = c_1$, and where $\pi(i) = i$ for $i\not\in\{c_1, \ldots, c_k\}$. (Note that this notation is very similar to notation we will later use to describe ordered lists. The two can be distinguished by noting that we do not use commas to describe a cycle permutation, but will use them to describe ordered lists.)

To define the \cds sorting operation, associate with each entry of the permutation $\pi\in \textsf{S}_n$ left and right pointers as follows: For an entry $k \in \{1,2, \ldots,n\}$ of $\pi$, the \emph{left pointer} of $k$ is $\langle k-1,\; k\rangle$, while the \emph{right pointer} of $k$ is $\langle k,\; k+1\rangle$. By convention, the smallest entry, $1$, does not have a left pointer, and the largest entry, $n$, does not have a right pointer.

\begin{example}\label{ex:cds1}
Equation (\ref{eq:pointerexample}) shows the permutation $\pi = [2 \ 4 \ 3 \ 1 \ 5]$ with all pointers marked.
\begin{equation} \label{eq:pointerexample}
\pi =[_{\langle1,2\rangle}2_{\langle 2,3\rangle} \ \ \ _{\langle 3,4\rangle}\!4_{\langle 4,5\rangle} \ \ \ _{\langle 2,3\rangle}3_{\langle 3,4\rangle} \ \ \  
1_{\langle 1,2\rangle} \ \ \ _{\langle 4,5\rangle}5
].
\end{equation}

\end{example}

Observe that each pointer in a permutation occurs twice. Given two pointers, $p$ and $q$, in the permutation $\pi$, the sorting operation \cds at these pointers acts as follows on $\pi$: If the pointers do not appear in the order $\cdots p\; \cdots q\; \cdots p\; \cdots q\; \cdots$ in $\pi$, then \cds does not apply and we say that the pointer context is invalid. Otherwise, the two segments of $\pi$ that are flanked by the pointer context $p\cdots q$ are interchanged.

\begin{repexample}{ex:cds1}
 The pointers $p = \langle 3,\; 4 \rangle$ and $q = \langle 4,\; 5 \rangle$ appear in $\cdots p\cdots q\cdots p \cdots q \cdots$ context in the permutation $\pi = \lbrack 2\; 4\; 3\; 1\; 5 \rbrack$. \cds applied to $\pi$ for this pointer context produces the permutation $\lbrack 2\;1\; 3\; 4\; 5 \rbrack$.
On the other hand, as the pointers $r = \langle 1,\; 2 \rangle$ and $s = \langle 3,\; 4 \rangle$ appear in $ \cdots r\cdots s\cdots s \cdots r \cdots$ context in $\pi$, \cds cannot be applied. \end{repexample}

When there are no pointers $p$ and $q$ that appear in context $\cdots p\cdots q\cdots p\cdots q \cdots$ in $\pi$, the permutation $\pi$ is said to be a \cds \emph{fixed point}. For each positive integer $n$, there are exactly $n$ \cds fixed points in $\textsf{S}_n$, namely the permutations $\lbrack k+1\;  \cdots \; n\; 1\; 2\; \cdots\; k\rbrack$ for $1\le k<n$, and the identity permutation $\lbrack 1\; 2\; \cdots \; n-1\; n\rbrack$.

By  \cite{AHMMSTW}, we know that for each permutation $\pi$ in $\textsf{S}_n$ that is not a \cds fixed point, some sequence of applications of \cds to $\pi$ terminate in a \cds fixed point. If a sequence of applications of \cds to the permutation $\pi$ terminates in the identity permutation $\lbrack 1\; 2\; \cdots\; n\rbrack$, we say that $\pi$ is \cds-\emph{sortable}. 
The \cds-sortability of permutations has been characterized in prior works such as \cite{AHMMSTW} and \cite{cdssort}. In \cite{AHMMSTW}, the obstacle to sortability of a permutation $\pi\in \textsf{S}_n$ is identified as follows. Suppose $\pi = \lbrack a_1\; a_2\; \cdots \; a_n\rbrack$. Define the cycle permutations $X_n$ and $Y_{\pi}$ by

\begin{equation}\label{eq:X_ndef}
 X_n := (0\; 1\; 2\; \cdots\; n), \text{ and }
\end{equation}
\begin{equation}\label{eq:Y_pidef}
 Y_{\pi} := (0\; a_n\; a_{n-1}\; \cdots\; a_1).
\end{equation}
Then define 
\begin{equation}\label{eq:C_pidef}
  C_{\pi} := Y_{\pi}\circ X_n.
\end{equation}
In equation (\ref{eq:C_pidef}) the symbol ``$\circ$" denotes functional composition, and we use the standard convention that $f\circ g(x)$ denotes the value $f(g(x))$.

When the entries $0$ and $n$ occur in the same cycle in the disjoint cycle decomposition of $C_{\pi}$, we shall write this cycle in the form
\begin{equation}\label{eqstrpile}
   (0\; u_1\; u_2\; \cdots\; u_j\; n\; b_1\; b_2\; \cdots\; b_k).
\end{equation}
The set $\textsf{SP}(\pi) = \{b_1, b_2, \ldots, b_k\}$ is said to be the \emph{strategic pile of $\pi$}. If $0$ and $n$ do not appear in the same cycle, we define $\textsf{SP}(\pi)$ to be the empty set. 
The ordered list $\textsf{SP}^*(\pi) = (b_1, b_2, \ldots, b_k)$ is called the \emph{ordered strategic pile of} $\pi$, and its ordering is determined by the order of appearance in (\ref{eqstrpile}). In \cite{AHMMSTW}, it was proven that a permutation $\pi$ is \cds-sortable if and only if its strategic pile is the empty set (i.e. if and only if 0 and $n$ do not appear in the same cycle).

\begin{example}\label{ex:strp1}
For the permutation $\pi = \lbrack 2\; 5\; 1\; 4\; 3 \rbrack$ we have $\textsf{C}_{\pi} = Y_{\pi} \circ X_5 = (0 \; 3\; 4\; 1\; 5\; 2)(0\; 1\; 2\; 3\; 4\; 5) = (0\; 5\; 3\; 1)(2\; 4)$, written in disjoint cycle form. Thus, the strategic pile of $\pi$ is the set $\textsf{SP}(\pi) = \{1, 3\}$, while $\textsf{SP}^*(\pi) = (3, 1)$.
\end{example}

The strategic pile of a permutation $\pi$ is intimately related to the set of achievable \cds fixed points:

\begin{thm}[\cite{AHMMSTW}]\label{thm:strpileth} If a permutation $\pi\in \textsf{S}_n$ is not \cds-sortable, then the following are equivalent for $1\le k<n$:
\begin{enumerate}
\item{There is a sequence of applications of \cds to $\pi$ that terminates in the \cds fixed point $\lbrack k+1\; k+2 \; \cdots\; n\; 1\; 2\; \cdots\; k \rbrack$.}
\item{$k$ is a member of the strategic pile of $\pi$.}
\end{enumerate}
\end{thm}

We now investigate the number of permutations in $\textsf{S}_n$ with maximum size strategic piles; these permutations can be considered to have maximal \cds non-sortability.

\section{Maximum Size Strategic Piles}

Since there are $n$ \cds fixed points (including the identity permutation), Theorem \ref{thm:strpileth} implies that a strategic pile of a permutation in $\textsf{S}_n$ can have at most $n-1$ elements.

\begin{lemma}\label{fullpileparity}
If there is a permutation in $\textsf{S}_n$ which has a strategic pile of size $n-1$, then $n$ is even.
\end{lemma}
\begin{proof}
By (\ref{eqstrpile}), if the strategic pile of permutation $\pi$ has size $n-1$, then
\begin{equation}\label{eq:fullstrpile}
C_{\pi} =    (0\; n\; b_1\; b_2\; \cdots\; b_{n-1}).
\end{equation}
But $C_{\pi}$ is the composition of two $(n+1)$-cycles, and thus an even permutation. Therefore $n$ is even. 
\end{proof}

As we shall see later, the converse of Lemma \ref{fullpileparity} also holds. As a consequence of Lemma \ref{fullpileparity} we find

\begin{cor}\label{oddmaxpile}
If $n$ is odd, then the strategic pile of an element of $\textsf{S}_n$ has at most $n-2$ elements.
\end{cor}
We shall also later see that there are permutations in $\textsf{S}_n$ with strategic pile of size $n-2$ for every odd integer $n \geq 1$. In the next two subsections we count for each $n$ the number of permutations in $\textsf{S}_n$ with strategic pile of maximal size for $n$. Subsection \ref{subsec:evenn} is dedicated to the case when $n$ is even, and Subsection \ref{subsec:oddn} is dedicated to the case when $n$ is odd.

\subsection{Maximum Size Strategic Piles for Even Values of $n$}\label{subsec:evenn}

\begin{thm} \label{fullStratPile}
For each even number $n$, the number of permutations in $\textsf{S}_n$ with strategic pile of size $n-1$ is $$\frac{2(n-1)!}{n}.$$
\end{thm}

As noted in the proof of Lemma \ref{fullpileparity}, an element of $\textsf{S}_n$ having a strategic pile of size $n-1$ is related to the possibility of factoring certain $(n+1)$-cycles into two $(n+1)$-cycles. As a result, to prove Theorem \ref{fullStratPile}, we first introduce some  additional notation, which we will use to define injective maps between sets of factorizations.

\begin{notation}
\hspace{1cm}
\begin{itemize}
\item Let $A$ denote the set of all factorizations of $X_{n-2}$ into two $(n-1)$-cycles.
\item Let $B$ denote the set of all factorizations of $X_n$ into two $(n+1)$-cycles where the right-most factor is of the form $(0 \; n \;  \cdots)$.
\item Let $B_i$ denote the subset of $B$ whose elements have right-most factors of form $(0 \; n \; i \; \cdots)$.
\item Define $\lambda_{n}=(0\;n\;1)$ and $c_{n}=(1\;2\;\cdots\;n-1)$.
\end{itemize}
\end{notation}

We will begin by constructing a bijection between the sets $A$ and $B_1$ in Lemmas \ref{firstTransform} and \ref{B1toAmap}. We will then show in Lemma \ref{circleTransform} that there is an injection from $B_i$ to $B_{i+1}$ for every $1 \leq i \leq n-1$, where this subscript addition is done modulo $n-1$, with $n-1$ as the additive identity; in other words, we will show $B_1 \rightarrow B_2, \; B_2 \rightarrow B_3, \ldots, \; B_{n-2} \rightarrow B_{n-1}, \text{ and } B_{n-1} \rightarrow B_1$, where $\rightarrow$ indicates an injective map. For a graphical depiction of these maps and sets, see Figure \ref{fig:diagram}. Since the injective maps between the $B_i$ sets form a cycle, it will follow that $B_i$ has the same cardinality for every $1 \leq i \leq n-1$. As a consequence, we will get that $\vert A \vert = \vert B_1 \vert = \cdots = \vert B_{n-1} \vert$. Finally, we can determine $|A|$ using the following prior result that counts the number of factorizations of an arbitrary $(n-1)$-cycle into two cycles of length $n-1$:

\begin{lemma}[\cite{cangelmi}] \label{cangelmi}
Let $\sigma \in \textsf{S}_{n-1}$ be an even $(n-1)$-cycle. Then the number of factorizations of $\sigma$ into two $(n-1)$-cycles is $$\frac{2(n-2)!}{n}.$$
\end{lemma}

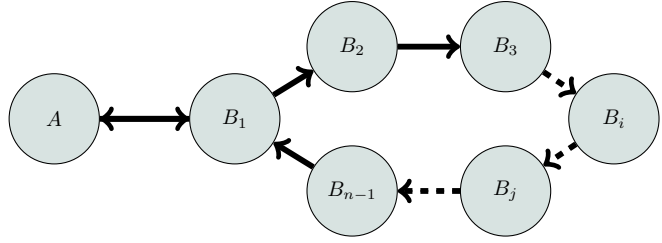
\begin{figure}
\centering
\captionsetup{width=.7\linewidth}
\captionsetup{justification=centering}
\begin{tikzpicture}[x = 4cm, y = 4cm, scale=.3]

       \node (L11) at (-2.5,1.5) [label]{};
       \node (L12) at (-2.5,-1.5)[label]{};
       \node (R11) at (-1.5,1.5)[label]{};
       \node (R12) at (-1.5,-1.5)[label]{};

       \vertex (c0) at (-2,0)[label,minimum size = 1.5cm,scale=.8,fill=lightgrey]{$A$}; 

       \node (L21) at (-0.5,1.5) [label]{};
       \node (L22) at (-0.5,-1.5)[label]{};
       \node (R21) at (3.0,1.5)[label]{};
       \node (R22) at (3.0,-1.5)[label]{};

       \vertex (c1) at (0,0)[label,minimum size = 1.5cm,scale=.8,fill=lightgrey] {$B_1$};
       \vertex (c2) at (1.3,0.8)[label,minimum size = 1.5cm,scale=.8,fill=lightgrey] {$B_2$};
       \vertex (c3) at (3.0,0.8)[label,minimum size = 1.5cm,scale=.8,fill=lightgrey] {$B_3$};
       \vertex (c4) at (4.2,0)[label,minimum size = 1.5cm,scale=.8,fill=lightgrey] {$B_i$};
       \vertex (c5) at (3.0,-0.8)[label,minimum size = 1.5cm,scale=.8,fill=lightgrey] {$B_j$};
       \vertex (c6) at (1.3,-0.8)[label,minimum size = 1.5cm,scale=.8,fill=lightgrey] {$B_{n-1}$};

\path


        (c0)  edge [style={->,double=black, thick}](c1)
        (c1) edge [style={->,double=black, thick}](c0)
        (c1) edge [style={->,double=black, thick}](c2)	
        (c2) edge [style={->, double=black, thick}](c3)
        (c3) edge [style={->, dashed, double=black, thick}](c4)
        (c4) edge [style={->,dashed, double=black, thick}](c5)	
       (c5) edge [style={->,dashed, double=black, thick}](c6)
       (c6) edge [style={->,double=black, thick}](c1);
\end{tikzpicture}
\caption{A depiction of the transformations introduced in Lemmas \ref{firstTransform}, \ref{B1toAmap}, and \ref{circleTransform}}
\label{fig:diagram}
\end{figure}

We now establish the previously described injections.

\begin{lemma} \label{firstTransform}
There is an injective map from $A$ to $B_1$.
\end{lemma}

\begin{proof}
Let $\gamma \circ \delta$ be a factorization in $A$. Namely, suppose 
$\gamma$ and $\delta$ are $(n-1)$-cycles satisfying $\gamma \circ \delta = X_{n-2}$. Define $\gamma_1$ and $\delta_1$ as follows:
\begin{align*}
\gamma_1 &:= \lambda_{n} \circ c_{n} \circ \gamma \circ (c_{n})^{-1}\\
\delta_1 &:= c_{n} \circ \delta \circ (c_{n})^{-1} \circ \lambda_{n}
\end{align*}
It suffices to show that $\gamma_1$ and $\delta_1$ are $(n+1)$-cycles, that $\delta_{1}$ is of the form $(0\;n\;1\;\cdots)$, and that $\gamma_1 \circ \delta_1 = X_n$.

Since conjugation preserves cycle structure, the factors $c_{n} \circ \gamma \circ (c_{n})^{-1}$ of $\gamma_1$ form an $(n-1)$-cycle with elements $\{0, 2, 3, \ldots, n-1\}$. Composing $\lambda_n$ with this $(n-1)$-cycle creates an $(n+1)$-cycle with elements $\{0, 1, 2, \ldots, n\}$.

Similarly, the factors $c_{n} \circ \delta \circ (c_{n})^{-1}$ of $\delta_1$ form an $(n-1)$-cycle with elements $\{0, 2, 3, \ldots, n-1\}$. Composing this $(n-1)$-cycle with $\lambda_n$ adds the elements $n$ and $1$ to form an $(n+1)$-cycle of the form $(0\;n\; 1\;\cdots)$.

Finally,
\begin{align*}
\gamma_1 \circ \delta_1 &= (\lambda_{n} \circ c_{n} \circ \gamma \circ (c_{n})^{-1}) \circ (c_{n} \circ \delta \circ (c_{n})^{-1} \circ \lambda_{n}) \\
&= \lambda_{n} \circ c_{n} \circ X_{n-2} \circ (c_{n})^{-1} \circ \lambda_{n} \\
&= \lambda_n \circ (0\;2\;3\;\cdots\;n-1) \circ \lambda_n = X_{n}.
\end{align*}
\end{proof}

\begin{example} \label{ex:maps}
Let $n=6$, which gives $X_{n-2} = X_4 = (0 \; 1 \; 2 \; 3 \; 4)$. Consider the factorization 
$$X_4 = (0 \; 1 \; 2 \; 3 \; 4) = \underbrace{(0 \; 2 \; 4 \; 1 \; 3)}_{\text{$\gamma$}} \underbrace{(0 \; 4 \; 3 \; 2 \; 1)}_{\text{$\delta$}}.$$
Using the maps defined in Lemma \ref{firstTransform}, we get
$$\gamma_1 = \lambda_6 \circ c_6 \circ \gamma \circ (c_6)^{-1} = (0 \; 6 \; 1)(1 \; 2 \; 3 \; 4 \; 5)(0 \; 2 \; 4 \; 1 \; 3)(5 \; 4 \; 3 \; 2 \; 1) = (0 \; 3 \; 5 \; 2 \; 4 \; 6 \; 1)$$
and 
$$\delta_1 = c_6 \circ \delta \circ (c_6)^{-1} \circ \lambda_6 = (1 \; 2 \; 3 \; 4 \; 5)(0 \; 4 \; 3 \; 2 \; 1)(5 \; 4 \; 3 \; 2 \; 1)(0 \; 6 \; 1) = (0 \; 6 \; 1 \; 5 \; 4 \; 3 \; 2).$$
Note that these are $(n+1)$-cycles, that $\delta_1$ is of the form $(0 \; n \; 1 \; \cdots)$, and that 
$$\gamma_1 \circ \delta_1 = (0 \; 3 \; 5 \; 2 \; 4 \; 6 \; 1)(0 \; 6 \; 1 \; 5 \; 4 \; 3 \; 2) = (0 \; 1 \; 2 \; 3 \; 4 \; 5 \; 6) = X_n,$$
as desired.
\end{example}

\begin{lemma} \label{B1toAmap}
There is an injective map from $B_1$ to $A$. 
\end{lemma}

\begin{proof}
Let $\gamma_1 \circ \delta_1 = (0 \; t_1 \; t_2 \; \cdots \; t_{n})(0 \; n \; 1 \; v_1 \; \cdots \; v_{n-2})$ be an arbitrary factorization in $B_1$. It suffices to show that we can recover from $\gamma_1$ and $\delta_1$ a factorization $\gamma \circ \delta$ of $X_{n-2}$ in $A$. Let $\delta := (c_{n})^{-1} \circ \delta_1 \circ (\lambda_{n})^{-1} \circ c_{n}$. Then,
\begin{align*}
\delta &= (c_{n})^{-1} \circ (0 \; n \; 1 \; v_1 \; \cdots \; v_{n-2}) \circ (\lambda_{n})^{-1} \circ c_{n} \\
&= (c_n)^{-1} \circ (0\; v_1 \; \cdots \; v_{n-2})(1)(n) \circ c_n \\ 
&= (0\; v_1-1\; \cdots \; v_{n-2}-1).
\end{align*}
It follows that $\delta$ is an $(n-1)$-cycle.

Since $\gamma_1 = (0\; t_1\; t_2 \; \cdots \; t_{n})$ and $\gamma_1 \circ \delta_1 = X_n$, we have that $t_n = 1$ and $t_{n-1}=n$. Let $\gamma = (c_{n})^{-1} \circ (\lambda_{n})^{-1} \circ \gamma_1 \circ c_{n}$. Then,
\begin{align*}
\gamma &= (c_n)^{-1} \circ (\lambda_n)^{-1} \circ (0 \; t_1 \; \cdots \; t_{n-2} \; n \; 1) \circ c_n = (c_n)^{-1} \circ (0\; t_1\; \cdots\; t_{n-2})(n)(1) \circ c_n.
\end{align*}
Since conjugation preserves cycle structure, $\gamma$ is an $(n-1)$-cycle.

Finally,
\begin{align*}
\gamma \circ \delta &= (c_n{}^{-1} \circ \lambda_n{}^{-1} \circ \gamma_1 \circ c_n) \circ (c_n{}^{-1} \circ \delta_1 \circ \lambda_n{}^{-1}\circ c_n)\\
&= c_n{}^{-1} \circ \lambda_n{}^{-1} \circ X_n \circ \lambda_n{}^{-1} \circ c_n = X_{n-2}.
\end{align*}

We have shown that $\gamma$ and $\delta$ are $(n-1)$-cycles and that $\gamma \circ \delta = X_{n-2}$. It follows that $\gamma \circ \delta \in A$, and this completes the proof.
\end{proof}

Since the injective maps defined in the proof of Lemma \ref{B1toAmap} are merely inverses of those defined in the proof of Lemma \ref{firstTransform}, these maps in fact serve as bijective maps between the sets $A$ and $B_1$. It follows that $\vert A \vert = \vert B_1 \vert$. The next lemma will function to show that $\vert B_1 \vert = \cdots = \vert B_{n-1} \vert$.

\begin{lemma} \label{circleTransform}
\hspace{1cm}
\begin{enumerate}
\item For every $1 \leq i \leq n-2$, there is an injection from $B_i$ to $B_{i+1}$.
\item There is an injection from $B_{n-1}$ to $B_1$.
\end{enumerate}
In other words, $B_1 \rightarrow B_2 \rightarrow \cdots \rightarrow B_{n-1} \rightarrow B_1$, where each $\rightarrow$ indicates an injective map.
\end{lemma}

\begin{proof} We prove the two statements separately. 

\noindent \textbf{Proof of (1).} Let $i$ satisfy $1 \leq i \leq n-2$. Let $\gamma_i$ and $\delta_i$ be $(n+1)$-cycles, where $\delta_i$ is of the form $(0 \; n \; i \; \cdots)$, and where $\gamma_i\circ \delta_i = X_n$. Let $r_{n}$ denote the cycle $(2 \; 1 \; n)$. Define
\begin{align*}
\gamma_{i+1} &= r_n \circ c_{n} \circ \gamma_i \circ (c_{n})^{-1}\\
\delta_{i+1} &= c_{n} \circ \delta_i \circ (c_{n})^{-1}
\end{align*}
It suffices to show that $\gamma_{i+1}$ and $\delta_{i+1}$ are $(n+1)$-cycles, that $\delta_{i+1}$ is of the form $(0 \; n \; i+1 \; \cdots)$, and that $\gamma_{i+1} \circ \delta_{i+1} = X_n$.

Since conjugation preserves the cycle structure of a permutation, both $\delta_{i+1} = c_{n} \circ \gamma_i \circ (c_{n})^{-1}$ and $c_{n} \circ \delta_i \circ (c_{n})^{-1}$ are $(n+1)$-cycles.
One can also check that composition with $r_n$ does not affect the cycle structure of $c_{n} \circ \gamma_i \circ (c_{n})^{-1}$, meaning $\gamma_{i+1}$ is also an $(n+1)$-cycle.

Next, observe that
$$\delta_{i+1}(0) = c_n(\delta_{i}(0)) = c_n(n) = n,$$ 
and 
$$\delta_{i+1}(n) = c_n(\delta_{i}(n)) = c_n(i) = i+1.$$ 
Therefore, $\delta_{i+1}$ is of the form $(0\;n\;i+1\;\cdots)$.

Finally, 
\begin{align*}
\gamma_{i+1} \circ \delta_{i+1} &= (r_{n} \circ c_{n} \circ \gamma_{i} \circ (c_{n})^{-1}) \circ (c_{n} \circ \delta_{i} \circ (c_{n})^{-1})\\ 
&= r_{n} \circ c_{n} \circ X_{n} \circ (c_{n})^{-1}\\
&= r_n \circ (0 \; 2 \; 3 \; \cdots \; n-1 \; 1 \; n)\\
&= X_{n}.
\end{align*}

\noindent \textbf{Proof of (2).} Statement (2) follows from the observations that $\delta_1 \neq \delta_i \neq \delta_n$ for all $1 < i < n$, and that $\delta_{n} = \delta_{1}$. The latter observation follows directly from the fact that the order of $c_n$ in the group of permutations is $n-1$.
\end{proof}

\begin{repexample}{ex:maps}
One can check that under the $\delta_i \rightarrow \delta_{i+1}$ map defined in the proof of Lemma \ref{circleTransform}, we get 
\begin{align*}
\delta_1 &= (0 \; 6 \; 1 \; 5 \; 4 \; 3 \; 2) \rightarrow (0 \; 6 \; 2 \; 1 \; 5 \; 4 \; 3) \rightarrow (0 \; 6 \; 3 \; 2 \; 1 \; 5 \; 4) \rightarrow (0 \; 6 \; 4 \; 3 \; 2 \; 1 \; 5) \rightarrow (0 \; 6 \; 5 \; 4 \; 3 \; 2 \; 1) \rightarrow (0 \; 6 \; 1 \; 5 \; 4 \; 3 \; 2) = \delta_n.
\end{align*}
\end{repexample}

We now prove the main result of this section, Theorem \ref{fullStratPile}:

\begin{proof}[Proof of Theorem \ref{fullStratPile}]
Since $X_n = Y_{\pi}^{-1} \circ C_{\pi}$, we count the factorizations of $X_n$ into two $(n+1)$-cycles where the second factor has the form $(0 \; n \; \cdots )$. This is the sum $ \sum_{1 \leq i \leq n-1} \vert B_i \vert$. By Lemma \ref{cangelmi}, there are $\displaystyle \frac{2(n-2)!}{n}$ factorizations of $X_{n-2}$ with two $(n-1)$-cycles. In other words, $\displaystyle \vert A \vert = \frac{2(n-2)!}{n}$. By Lemmas \ref{firstTransform}, \ref{B1toAmap}, and \ref{circleTransform}, we have that $\vert A \vert = \vert B_1 \vert = \cdots = \vert B_{n-1} \vert$. It follows that for each even $n$, the number of permutations in $\textsf{S}_n$ with strategic pile size $n-1$ is 
$$\sum_{1 \leq i \leq n-1} \vert B_i \vert = (n-1) \vert A \vert = \frac{2(n-1)!}{n},$$
as desired.

\end{proof}

We will now prove a similar result for $n$ odd.

\subsection{Maximum Size Strategic Piles for Odd Values of $n$}\label{subsec:oddn}

\begin{thm}\label{thm:maxstrpileodd}
For each odd number $n$, the number of permutations in $\textsf{S}_n$ with strategic pile size $n-2$ is $2(n-2)!$.
\end{thm}

A permutation $\pi = \lbrack a_1\; a_2\; \cdots \; a_n\rbrack$ is said to have an \emph{adjacency} if there is an index $i<n$ such that $a_{i+1} = a_i+1$.

\begin{lemma}\label{lemma:oddntoeven}
Let $n>1$ be an odd number and let $\pi$ be an element of $\textsf{S}_n$. If $\pi$ has a strategic pile of cardinality $n-2$, then $\pi$ has a single adjacency.
\end{lemma}
\begin{proof} 
Let $\pi = \lbrack a_1\; a_2\; \cdots \; a_n\rbrack$. We have that
\[
  C_{\pi} = (0\; a_n\; \cdots\; a_1)\circ (0\; 1\; \cdots n),
\]
which is an even permutation. It follows from Lemma \ref{fullpileparity} that $C_{\pi}$ is not a single cycle. Since $\pi$ has a strategic pile of cardinality $n-2$, we have that $C_{\pi}$ is of the form
\[
  C_{\pi} = (0\; n\; c_1\; \cdots\;c_{n-2})\circ(x).
\]
The singleton cycle $(x)$ comes about on account of the following configuration in the computation of $C_{\pi}$:
\[
 (\cdots \; x+1\; x\; \cdots) \circ (0\; 1\; \cdots\; x\; x+1\; \cdots \; n).
\]
Thus, in $\pi$ we have that for some $i$, $a_i = x$ and $a_{i+1} = x+1$.
\end{proof}

When a permutation $\pi$ in $\textsf{S}_n$ has a single adjacency, it can be projected to a unique corresponding permutation $P(\pi)$ in $\textsf{S}_{n-1}$ which has no adjacencies, as follows: Let $\pi = \lbrack a_1\; a_2\; \cdots \; a_i\; a_{i+1}\; \cdots \; a_n\rbrack \in \textsf{S}_n$ have the single adjacency $a_{i+1} = a_i+1$. We define $P(\pi)$ by removing the second element of the adjacency and reducing all large elements by 1. More precisely, $P({\pi}) = \lbrack a_1'\; a_2'\; \cdots\; a_{n-1}'\rbrack$, where
\[
  a_j' = \left\{\begin{tabular}{ll}
              $a_j$ & if $j\le i$ and $a_j\le a_i$\\
              $a_j-1$ & if $j\le i$ and $a_j>a_i$\\
              $a_{j+1}$ & if $n>j > i$ and $a_{j+1}<a_{i+1}$\\
              $a_{j+1}-1$ & if $n>j > i$ and $a_{j+1}\ge a_{i+1}$
             \end{tabular}
             \right.
\]

\begin{example}\label{ex:adj1}
The permutation $\pi = \lbrack 2\; 3\; 6\; 1\; 5\; 4 \rbrack$ has one adjacency. $P(\pi) = \lbrack 2\; 5\; 1\; 4\; 3 \rbrack$ has no adjacencies. Observe that there are $5$ different elements of $\textsf{S}_6$, each with a single adjacency, that give rise in this way to $\lbrack 2\; 5\; 1\; 4\; 3 \rbrack$, namely:
$\lbrack 2\; {\underline 3}\; 6\; 1\; 5\; 4 \rbrack$, $\lbrack 2\; 5\; {\underline 6}\; 1\; 4\; 3 \rbrack$, $\lbrack 3 \; 6\; 1\;  {\underline 2}\;  5\; 4 \rbrack$, $\lbrack 2\; 6\; 1\; 4\; {\underline 5}\; 3 \rbrack$, and $\lbrack 2\;  6\; 1\; 5\; 3\; {\underline 4}\; \rbrack$. 
\end{example}

\begin{observation}\label{lemma:projection}
If $n>1$ is an odd number and $\pi\in\textsf{S}_n$ is a permutation with a strategic pile of cardinality $n-2$, then $P({\pi})\in\textsf{S}_{n-1}$ is a permutation with a strategic pile of cardinality $n-2$.
\end{observation}

Conversely, if we are given a permutation $\mu \in \textsf{S}_{n-1}$ which has no adjacencies, say $\lbrack a_1'\; a_2'\; \cdots \; a_{n-1}'\rbrack$, and any position $i$, we can construct a unique permutation $E(\mu,i) = \lbrack a_1\; a_2\; \cdots\; a_n\rbrack$ in $\textsf{S}_n$ which has a single adjacency, and for which $P(E(\mu,i)) = \mu$: Namely, define $a_{i+1}$ to be $a_i'+1$; for $j<i$ define $a_j = a_j'+1$ if $a_i<a_j$, and $a_j = a_j'$ otherwise; for $j>i$ define $a_{j+1} = a_j'$ if $a_j'<a_i'$, and $a_{j+1} = a_j'+1$ otherwise.

\begin{observation}\label{lemma:expansion}
If $n>1$ is an odd number and $\mu \in\textsf{S}_{n-1}$ is a permutation with a strategic pile of cardinality $n-2$, and if $i\le n-1$, then $E({\mu},i)\in\textsf{S}_{n}$ is a permutation with a strategic pile of cardinality $n-2$.
\end{observation}

With these facts at our disposal we now prove Theorem \ref{thm:maxstrpileodd}: 
\begin{proof}[Proof of Theorem \ref{thm:maxstrpileodd}]
By Observation \ref{lemma:expansion} each permutation $\mu \in\textsf{S}_{n-1}$ with full strategic pile produces $n-1$ permutations $\pi_i = E(\mu,i)$ for $i\le n-1$ in $\textsf{S}_{n}$ with strategic pile of size $n-2$. Thus by Theorem \ref{fullStratPile} there are at least $\displaystyle (n-1) \cdot \frac{2(n-2)!}{n-1} = 2(n-2)!$ elements of $\textsf{S}_n$ with strategic pile of size $n-2$. Conversely, by Lemma \ref{lemma:oddntoeven} each element of $\textsf{S}_n$ that has a strategic pile of size $n-2$ arises in this way.
\end{proof}

\section{Strategic Piles of Size $k$} \label{sec:pilesizek}

Having quantified the number of permutations with maximum size strategic piles, we next produce an analogous quantification for permutations with strategic piles of arbitrary size. Before stating the main result of this section, we first establish terminology and structural properties of permutations with strategic piles of size $k$. 

\subsection{Structure of Permutations with Strategic Pile of Size $k$}

\begin{prop}
\label{form}
For a permutation $\pi$ in $\textsf{S}_n$, $\textsf{SP}^{\;*}(\pi)=(b_1, b_2, \ldots, b_k)$ if and only if the following are true: 
\begin{enumerate}
\item $\pi(1) = b_k+1$. 
\item $\pi(n) = b_1$.
\item For all $j \in \{2, 3, \ldots, k-1\}$, the element $b_j$ appears to the immediate left of $b_{j-1}+1$ in $\pi$.
\end{enumerate}
\end{prop}

\begin{proof}
First note that $C_{\pi}(b_k) =0$ if and only if $Y_{\pi}(b_k+1)=0$, since $C_{\pi}(b_k)= Y_{\pi}(X(b_k))=Y_{\pi}(b_k+1)$. Also, by definition, $Y_{\pi}(b_k+1)=0$ if and only if $\pi(1) = b_k+1$. Therefore, $C_{\pi}(b_k)=0$ if and only if $\pi(1) = b_k + 1$.

Secondly, $C_{\pi}(n) = b_1$ if and only if  $Y_{\pi}(0) = b_1$, since $C_{\pi}(n) = Y_{\pi}(X(n)) = Y_{\pi}(0)$. Also, by definition, $Y_{\pi}(0) = b_1$ if and only if $\pi(n) = b_1$. Therefore, $C_{\pi}(n) = b_1$ if and only if $\pi(n) = b_1$.

Finally, for $j\in \{2, 3, \ldots, k-1\}$, $C_{\pi}(b_{j-1}) = b_j$ if and only if $Y_{\pi}(b_{j-1} + 1) = b_j$, since $C_{\pi}(b_{j-1}) = Y_{\pi}(X(b_{j-1})) = Y_{\pi}(b_{j-1} + 1)$. Also, by definition, $Y_{\pi}(b_{j-1} + 1) = b_j$ if and only if $b_j$ immediately precedes $b_{j-1} + 1$ in $\pi$. Therefore, $C_{\pi}(b_{j-1}) = b_j$ if and only if $b_j$ appears immediately to the left of $b_{j-1} +1$ in $\pi$.

Since $\textsf{SP}^*(\pi)=(b_1, b_2, \ldots, b_k)$ if and only if $C_{\pi}(b_k) =0$, $C_{\pi}(n) = b_1$, and $C_{\pi}(b_{j-1})=b_j$ for all $j\in\{2, 3, \ldots, k-1\}$, our proposition holds.
\end{proof}

With $b_j$ denoting the $j$-th element of the ordered strategic pile of a permutation $\pi$, adjacent entries of the form $b_{j} \; b_{j-1}+1$ in $\pi$ are called a \emph{pair}. Viewing subscripts modulo $k$, we also consider $b_1 \; b_k+1$ a pair. In general, a permutation $\pi$ with $\textsf{SP}^*(\pi) = (b_1, b_2, \ldots, b_k)$ has the following form in terms of its pairs:
\begin{equation}\label{pairform}
[b_k+1 \; \cdots \; b_{x_1} \; b_{x_1-1}+1 \; \cdots \; b_{x_2} \; b_{x_2-1}+1 \; \cdots \cdots \; b_{x_{k-1}} \; b_{x_{k-1}-1}+1 \; \cdots \; b_1].
\end{equation}

\begin{defn} \label{def:orderedpairlist}
The ordered list
\[
  \sigma_{\pi} = (b_{x_1}, b_{x_2}, \ldots, b_{x_{k-1}}, b_1),
\]
consisting of the first member of each pair, in the order of occurrence in $\pi$, is said to be the \textit{ordered pair list} of $\pi$.
\end{defn}

Since $b_1$ is the final entry of a permutation $\pi$ with a nonempty strategic pile, $b_1$ is always the terminating member of the ordered pair list $\sigma_{\pi}$.

\begin{example}\label{ex:mergepairduplication}
The permutation $\pi = \lbrack 6 \ 4 \ 5 \ 8 \ 7 \ 2 \ 3 \ 1 \rbrack$ has strategic pile $\textsf{SP}(\pi) = \{1, 5, 7\}$, and $\textsf{SP}^{*}(\pi) = (1,7,5) = (b_1, b_2, b_3)$. Therefore, $\pi = [6 \; 4 \; b_3 \; 8 \; b_2 \; 2 \; 3 \; b_1] = [b_3+1 \; 4 \; b_3 \; b_2+1 \; b_2 \; b_1+1 \; 3 \; b_1]$, as suggested by Proposition \ref{form}. This gives that $\sigma_{\pi} = (b_3,b_2,b_1)$.
\end{example}

In Definition \ref{def:orderedpairlist} we defined the ordered pair list with respect to a specified permutation $\pi$. Note, however, that we can instead define an ordered pair list independently of a specific permutation. Using this interpretation, any permutation where the $x_i$-th strategic pile element leads the $i$-th pair for all $1 \leq i \leq k-1$ will be said to have the ordered pair list $\sigma = (b_{x_1}, b_{x_2}, \ldots, b_{x_{k-1}}, b_1)$.

\begin{example}
Consider the ordered pair list $\sigma = (b_2,b_3,b_1)$, defined independently of a specific permutation. Any permutation with $\textsf{SP}^* = (b_1, b_2, b_3)$ that is of the form 
$$[b_3+1 \; \cdots \; b_2 \; b_1+1 \; \cdots \; b_3 \; b_2+1 \; \cdots \; b_1 ]$$
will have ordered pair list $\sigma$. In particular, the permutation $\pi = [2 \; 3 \; 6 \; 1 \; 4 \; 5]$ has $\textsf{SP}^* = (5,3,1)$ and thus $\sigma_{\pi} = (3,1,5) = (b_2, b_3, b_1) = \sigma$. Similarly, the permutation $\nu = [4 \; 1 \; 6 \; 3 \; 2 \; 5]$ has $\textsf{SP}^* = (5,1,3)$ and thus $\sigma_{\nu} = (1,3,5) = (b_2, b_3, b_1) = \sigma$.
\end{example}

As Example \ref{ex:formexample} will illustrate, for subsequent pairs $b_{x_i} \; b_{x_i-1}+1$ and $b_{x_{i+1}} \; b_{x_{i+1}-1}+1$ of a permutation $\pi$ it may happen that $b_{x_i-1}+1 = b_{x_{i+1}}$, in which case $b_{x_i}$ and $b_{x_{i+1}}$ are consecutive entries of $\pi$. As these adjacencies will be of central importance in the proof of Theorem \ref{pileSizeK}, we formalize their definition as follows:

\begin{definition} \label{def:merge}
An adjacency of strategic pile members $b_{x_i}$ and $b_{x_{i+1}}$ in $\pi$ is said to be a \textit{merge between $b_{x_i}$ and $b_{x_{i+1}}$} in $\pi$. Such a merge will be denoted $b_{x_i} \; b_{x_{i+1}}$.
\end{definition}

\begin{example} \label{ex:formexample}
The permutation $\pi = [5\;4\;6\;3\;2\;1]$ has strategic pile $\textsf{SP}(\pi) = \{1, \; 3, \; 4, \; 5\}$, and $\textsf{SP}^*(\pi) = (1, \; 3, \; 5, \; 4) = (b_1, \; b_2, \; b_3, \; b_4)$. Moreover, $\sigma_{\pi} = (b_3,\; b_4,\; b_2,\;b_1)$ since $\pi$ has the form $\lbrack b_3\; b_4\; 6\; b_2\; 2\; b_1\rbrack$. The strategic pile members $b_3$ and $b_4$ are adjacent in $\pi$, and thus there is a merge in $\pi$. Since we are also considering $b_1\;b_4+1$ a pair in $\pi$, $b_1\; b_3$ is also ruled a merge in $\pi$.
\end{example}

When considering strategic piles of size $k$, we refer to an arrangement of the strategic pile variables $b_i$ and $b_i+1$ for $1 \leq i \leq k$ as a \emph{strategic pile variable arrangement} if the arrangement satisfies the properties described in Proposition \ref{form}. All possible strategic pile variable arrangements can be obtained by shifting and merging pairs within the possible frameworks of the form (\ref{pairform}).

\begin{example}
The following are five of the possible strategic pile variable arrangements for permutations in $\textsf{S}_7$ with $\textsf{SP}^*=(b_1, b_2, b_3)$ and ordered pair list $\sigma = (b_2, b_3, b_1)$, where the \underline{\hspace{.5cm}}'s can be filled in by any remaining permutation elements:

\begin{enumerate}
\item $[b_3+1 \; \underline{\hspace{.5cm}} \; b_2 \; b_1+1 \; b_3 \; b_2+1 \; b_1]$ (no merges)
\item $[b_3+1 \; b_2 \; b_1+1 \; \underline{\hspace{.5cm}} \; b_3 \; b_2+1 \; b_1]$ (no merges)
\item $[b_2 \; b_1+1 \; \underline{\hspace{.5cm}} \; b_3 \; b_2 + 1 \underline{\hspace{.5cm}} \; b_1]$ (merge $b_1 \; b_2$)
\item $[b_2 \; b_1+1 \; b_3 \; b_2 + 1 \; \underline{\hspace{.5cm}} \; \underline{\hspace{.5cm}} \; b_1]$ (merge $b_1 \; b_2$)
\item $[b_2 \; b_1+1 \; \underline{\hspace{.5cm}} \; \underline{\hspace{.5cm}} \; \underline{\hspace{.5cm}} \; b_3 \; b_1]$ (merges $b_3 \; b_1$ and $b_1 \; b_2$)
\end{enumerate}
\end{example}

The above definitions and structural properties regarding permutations with strategic piles of size $k$ yield an approach to quantifying such permutations. Since a permutation has strategic pile size $k$ if and only if it takes the form described in Proposition \ref{form}, we start by counting the number of strategic pile variable arrangements. To this end, we define \emph{merge numbers}.

\begin{definition}\label{def:mergenumber} 
Consider the set of permutations $\pi \in \textsf{S}_n$ with $\textsf{SP}^*(\pi)=(b_1, \ldots, b_k)$. Given $\ell \geq 0$, the symbol $c_{k,\ell}$ denotes the number of ways to choose an ordered pair list $\sigma_{\pi}$ along with $\ell$ merges. The number $c_{k,\ell}$ is said to be a \textit{merge number}.
\end{definition}

\begin{example} \label{ex:mergenoexample}
For permutations with $\textsf{SP}^* = (b_1, b_2, b_3)$, the only possible ordered pair lists are $(b_2, b_3, b_1)$ and $(b_3, b_2, b_1)$, which correspond to the following permutation structures:
\begin{enumerate}
\item $\lbrack b_3+1 \; \cdots \; b_2 \; b_1+1 \; \cdots \; b_3 \; b_2+1 \; \cdots \; b_1 \rbrack$

\item $\lbrack b_3+1 \; \cdots \; b_3 \; b_2+1 \; \cdots \; b_2 \; b_1+1 \; \cdots \; b_1 \rbrack$
\end{enumerate}

In the first form, each of the merges $b_2 \; b_3$, $b_3 \; b_1$, and $b_1 \; b_2$ are possible, so there are three ways to create a single merge with this ordered pair list. In the second form, a merge cannot occur at all, since it would require that $b_i + 1 = b_i$, which is impossible. Therefore, $c_{3,1}= 1 \cdot 3 + 1 \cdot 0 = 3$. 
\end{example}

We are now ready to state the main result of this section.

\subsection{Main Result}

\begin{thm} \label{pileSizeK}
For $1 \leq k \leq n-1$ and even $n$, or $1 \leq k \leq n-2$ and odd $n$, the number of permutations in $\textsf{S}_n$ with strategic pile size $k$ is $$(n-k)! \sum \limits_{i=0}^{\infty} c_{k,i}\binom{n-(k+1)}{k-(i+1)}.$$ 
\end{thm}

As there is a limit on the number of merges that can occur in a permutation, each merge number $c_{k,i}$ will be zero for all $i$ above a certain value. We leave determining this maximum number of merges, as well as the general method for computing merge numbers, to Section \ref{mergenumbersection}. To prove Theorem \ref{pileSizeK}, we will 

\begin{itemize}
\item use merge numbers to determine the number of strategic pile variable arrangements (see Lemma \ref{lem:groupingplacement} and Corollary \ref{cor:grouping}), and
\item determine the number of ways to assign numerical values to the resulting variable arrangements (see Lemma \ref{lem:filler}).
\end{itemize}

Assuming we can compute each merge number $c_{k,i}$, we can suppose we are given a framework comprised of an ordered pair list and a set of merges. To quantify the possible strategic pile variable arrangements, we are left to account for how this framework can shift within $n$ positions. To this end, we develop terminology to refer to the components of this framework.  

\begin{example} \label{ex:grouping}
Consider an ordered pair list $\sigma = (b_{x_1}, \; b_{x_2}, \; \ldots, \; b_{x_{k-1}}, \; b_1)$, which by Lemma \ref{form} yields a permutation of the form
\[ 
[\text{$b_k+1$} \; \cdots \; b_{x_1} \; \text{\underline{$b_{x_1-1}+1$}} \; \cdots \; b_{x_2} \; \text{\underline{$b_{x_2-1}+1$}} \; \cdots \cdots \; b_{x_{k-1}} \; \text{\underline{$b_{x_{k-1}-1}+1$}} \; \cdots \; \text{$b_1$}].
\]
After a merge, say between $b_{x_1}$ and $b_{x_2}$, we get 
\[ 
[\text{$b_k+1$} \; \cdots \; b_{x_1} \; \text{\underline{$b_{x_2}$}} \; \text{\underline{$b_{x_2-1}+1$}} \; \cdots \cdots \; b_{x_{k-1}} \; \text{\underline{$b_{x_{k-1}-1}+1$}} \; \cdots \; \text{$b_1$}].
\]
\end{example}

In Example \ref{ex:grouping} above, observe that it may not be intuitive to call $b_{x_1} \; b_{x_2} \; b_{x_2 - 1} + 1$ a pair; we use the term \emph{grouping} to refer to pairs as well as any set of pairs joined by merges.

Recall that $b_k+1$ and $b_1$ are always in the first and last positions of a permutation, respectively. Moreover, observe that the position of each underlined element in Example \ref{ex:grouping} is determined by the placement of the leftmost element in its grouping. We call both of these types of elements \emph{determined}. 

In a permutation with strategic pile of size $k$ with no merges, there are $k+1$ determined elements (i.e. $b_1$, $b_k+1$, and $b_{j-1}+1$ for $2 \leq j \leq k$). Furthermore, observe that ``merging'' groupings does not affect the total number of determined elements, since a merge has the effect of equating a determined element with an undetermined element. In Example \ref{ex:grouping}, the merge between $b_{x_1}$ and $b_{x_2}$ equates $b_{x_1-1}+1$ (a determined element) with $b_{x_2}$ (an undermined element), making a grouping with two determined elements, the same total number that the pairs $b_{x_1} \; b_{x_1-1}+1$ and $b_{x_2} \; b_{x_2-1}+1$ had to begin with. Therefore, any grouping arrangement, despite the number of merges, will have $k+1$ determined elements. 

\begin{lemma} \label{lem:groupingplacement}
Given an ordered pair list $\sigma = (b_{x_1},\ldots, b_{x_k})$ and a set of $i$ merges, there are 
$$\binom{n-(k+1)}{(k-1)-i}$$
ways to place the resulting groupings within a permutation of length $n$.
\end{lemma}

\begin{proof}
Recall that in a permutation with strategic pile size $k$, there are always $k+1$ determined elements. For each determined element, we set aside one space in the permutation. This leaves $n-(k+1)$ unoccupied spaces in which to place the groupings. Since the leftmost variable of each grouping is the only undetermined variable in the grouping, we must only place these $(k-1)-i$ undetermined variables, and the placement of all other variables follows. Because there are $n-(k+1)$ spaces in which to place these undetermined variables, $\displaystyle \binom{n-(k+1)}{(k-1)-i}$ represents the number of ways to place the groupings. 
\end{proof}

\begin{cor} \label{cor:grouping}
The number of strategic pile variable arrangements in the set of permutations in $\textsf{S}_n$ with strategic piles of size $k$ and $i$ merges is
$$c_{k,i} \cdot \binom{n-(k+1)}{k-i-1}.$$
\end{cor}

\begin{proof}
By definition of the merge number $c_{k,i}$, there are $c_{k,i}$ ways to choose an ordered pair list $\sigma = (b_{x_1},\ldots,b_{x_k})$ along with $i$ merges. Each fixed ordered pair list and set of merges defines a framework of groupings, which by Lemma \ref{lem:groupingplacement} can be placed $\displaystyle \binom{n-(k+1)}{k-i-1}$ different ways among $n$ positions. Thus, there are a total of $\displaystyle c_{k,i} \binom{n-(k+1)}{k-i-1}$ different strategic pile variable arrangements.  
\end{proof}

We now determine the number of ways to assign a numerical value to each variable and to the remaining elements of the permutation. 

\begin{lemma} \label{lem:filler}
Given a fixed strategic pile variable arrangement (i.e. an ordered pair list $\sigma$ of strategic pile elements $\{ b_1, \ldots, b_k \}$, a set of merges, and a fixed set of grouping positions), there are $(n-k)!$ permutations in $\textsf{S}_n$ with that arrangement.
\end{lemma}

\begin{proof}
Given a strategic pile variable arrangement, the only thing left to do is assign values to the variables comprising the permutation. For each $b_j \in \textsf{SP}$, there is some variable $b_j+1$, whose value follows immediately from a value assignment of $b_j$. Therefore, only $n-k$ values need to be assigned, and there are $(n-k)!$ possible assignments.
\end{proof}

We are now ready to prove our main result of this section.

\begin{proof}[Proof of Theorem \ref{pileSizeK}]
By Corollary \ref{cor:grouping}, there are $\displaystyle c_{k,i} \binom{n-(k+1)}{k-i-1}$ strategic pile variable arrangements in the set of permutations in $\textsf{S}_n$ with strategic piles of size $k$ and $i$ merges. Summing over the number of merges $i$, we get that the total number of strategic pile variable arrangements for permutations in $\textsf{S}_n$ with strategic piles of size $k$ is
$$\sum_{i=0}^{\infty} c_{k,i} \binom{n-(k+1)}{k-i-1}.$$
We have by Lemma \ref{lem:filler} that there are $(n-k)!$ permutations corresponding to each strategic pile variable arrangement. It follows that there are
$$(n-k)!\sum_{i=0}^{\infty} c_{k,i} \binom{n-(k+1)}{k-i-1}$$
permutations in $\textsf{S}_n$ with strategic pile size $k$.
\end{proof}

\section{Determining the Values of Merge Numbers} \label{mergenumbersection}

As mentioned in Section \ref{sec:pilesizek}, there is a limit to the number of merges that can occur in a permutation with strategic pile of size $k$. Thus, there exists an $i$ such that $c_{k,i'} = 0$ for all $i' \geq i$. In this section we determine this $i$ and derive an algorithm for computing merge numbers. As described later this section, determining the efficiency of this algorithm is dependent on the solutions of certain open problems. However, we can explicitly compute merge numbers in some limited cases (see Table \ref{table:DiscoveredMergeNumbers}). Using these merge numbers, Theorem \ref{pileSizeK} gives the formulas given in Table \ref{table:DiscoveredFormulas}.

\begin{table}
\centering
\captionsetup{width=.7\linewidth}
\captionsetup{justification=centering}
\renewcommand{\arraystretch}{1.5}
\begin{tabular}{|c|rrrrrrcrc|c|} \hline
\textbf{$k$} & 1   & 2  & 3  & 4   & 5     &  6     & $\cdots$ & $k$ & $\cdots$ & OEIS                       \\ \hline
$c_{k,0}$    &  0! &1!  & 2! & 3!  & 4!    &  5!     &                & (k-1)!     &               &                  A000142    \\
$c_{k,1}$    &      &     & 3  & 16 & 90   & 576  &                & k(k-2)(k-2)!        &               &   A130744  \\
$c_{k,2}$    &      &     & 3  & 16 & 130 & 1116 &                & ?      &               &                                \\
$c_{k,3}$    &      &     &     &      &   80 & 1080 &               & ?      &               &                                \\
$c_{k,4}$    &      &     &     &      &   90 &    540  &              & ?      &               &                                \\
$\vdots$      &     &      &     &      &        &                          &         &               &     &  \\
$c_{k,i}$      &     &      &     &      &        &                          &         &               &      & \\  \hline
\end{tabular}
\caption{Merge Numbers Found Using Ad Hoc Methods}
\label{table:DiscoveredMergeNumbers}
\end{table}

\begin{table}
\centering
\captionsetup{width=.7\linewidth}
\captionsetup{justification=centering}
\renewcommand{\arraystretch}{1.5}
\begin{tabular}{|l|l|l|} \hline
\textbf{$k$} & \textbf{Number of Elements of $\textsf{S}_n$ with Strategic Piles of Size $k$}  & OEIS                       \\ \hline
1            & $(n-1)!\lbrack\binom{n-2}{0}0!\rbrack$              &      A000142                            \\
2            & $(n-2)!\lbrack\binom{n-3}{1}1!\rbrack$              &      A062119                            \\
3            & $(n-3)!\lbrack\binom{n-4}{2}2!+\binom{n-4}{1}3+\binom{n-4}{0}3\rbrack$    &  A267323 \\
4            & $(n-4)!\lbrack\binom{n-5}{3}3!+\binom{n-5}{2}16+\binom{n-5}{1}16\rbrack$ & A267324 \\
5            & $(n-5)!\lbrack\binom{n-6}{4}4!+\binom{n-6 }{3}90+\binom{n-6}{2}130+\binom{n-6}{1}80+\binom{n-6}{0}90\rbrack$   &  A267391\\
6            & $(n-6)!\lbrack\binom{n-7}{5}5!+\binom{n-7}{4}576+\binom{n-7}{3}1116+\binom{n-7}{2}1080+\binom{n-7}{1}540\rbrack$ & A281259\\
$\vdots$ & & \\
$k$          & $(n-k)!\lbrack\binom{n-(k+1)}{k-1}(k-1)!+\binom{n-(k+1)}{k-2}c_{k,1}+\cdots+\binom{n-(k+1)}{1}c_{k,k-2}\rbrack$ (k odd)  & \\ \hline
\end{tabular}
\caption{Known Formulas for the Number of Permutations with Strategic Piles of Size $k$}
\label{table:DiscoveredFormulas}
\end{table}

We now discuss how to compute merge numbers in general. First, however, we define two graph theoretic tools, which will be useful for accomplishing both of the aforementioned goals.

\subsection{Merge Graphs and $\tau$-Graphs} \label{3.3:mergestuff}

As in Example \ref{ex:mergenoexample}, the ordered pair list determines the set of possible merges for permutations with that ordered pair list. As a result, it will often be of use to classify permutations based on ordered pair list.

\begin{definition}
Let $\sigma$ be an ordered pair list. Define $T_{\sigma} := \{ \pi \in \textsf{S}_n \; \vert \;  \sigma_{\pi} = \sigma \}$.
\end{definition}

Moreover, to count the number of ways to choose an ordered pair list along with $\ell$ merges (i.e. to compute merge numbers), it is crucial that we are first able to determine the set of allowable merges corresponding to an ordered pair list. To this end, we define the following: 

\begin{definition}\label{psiandsigmacycles}
Consider the set of permutations in $\textsf{S}_n$ with ordered strategic pile $\textsf{SP}^*(\pi)=(b_1, b_2, \ldots, b_k)$ and ordered pair list $\sigma = (b_{x_1}, b_{x_2}, \ldots, b_{x_{k-1}}, b_1)$. Define
\begin{itemize}
\item{$\psi := (b_1 \; b_2 \; \cdots \; b_k)$,}
\item{$\sigma^*:=(b_{x_1}\; b_{x_2}\;\cdots\; b_{x_{k-1}}\; b_1)$, and}
\item{$\tau_{\sigma} := \sigma^{*} \circ \psi$.}
\end{itemize}
\end{definition}

Observe that there exists a permutation $\pi \in T_{\sigma}$ with $b_i +1 = b_j$ if and only if $\tau_{\sigma}(b_i) = b_j$. In other words, $\tau_{\sigma}$ describes exactly the allowable merges for the set of permutations $T_{\sigma}$. It will often be useful for us to encode this information graphically.

\begin{definition}\label{taugraph}
The $\tau$\textit{-graph} corresponding to an ordered pair list $\sigma$ is an at most in-degree one, out-degree one directed graph $\mathscr{T}_{\sigma} = (V, \; E)$, where $V = \{ b_1, b_2, \ldots, b_k \}$ and $(b_i, b_j) \in E$ if and only if $b_i + 1 = b_j$ in some permutation in $T_{\sigma}$. 
\end{definition}

Note that an edge $(b_i,b_j)$ in a $\tau$-graph $\mathscr{T}_{\sigma}$ corresponds to the equality $b_i + 1 = b_j$, and not to the merge $b_i \; b_j$. Rather, the edge $(b_i,b_j)$ represents the existence of a merge $b_{i+1} \; b_j$ in some permutation $\pi \in T_{\sigma}$. Furthermore, it may be useful to note that $\tau$-graphs are comprised completely of cycles and isolated vertices, and that each cycle in the $\tau$-graph $\mathscr{T}_{\sigma}$ corresponds to a cyclic factor of the cycle permutation $\tau_{\sigma}$.

\begin{example}
Consider the set of permutations $T_{\sigma}$ corresponding to the ordered pair list $\sigma = (b_1, b_6, b_7, b_5, b_2, \allowbreak b_4, b_3)$. Then
\[
\tau_{\sigma} = \sigma^* \circ \psi =  (b_1 \; b_6 \; b_7 \;b_5 \;b_2 \;b_4 \;b_3)\circ (b_1\; b_2\; b_3\; b_4\; b_5\; b_6\; b_7) = (b_1\; b_4\; b_2)(b_5\; b_6\; b_7)(b_3), 
\]
which indicates that there are two cycles in $\mathscr{T}_{\sigma}$ formed by the edges $\{ (b_1, \; b_4), \; (b_4, b_2), \; (b_2, b_1) \}$ and $\{ (b_5, \; b_7), \; (b_7, \; b_6), \; (b_6, b_5) \}$. Figure \ref{fig:mergeandtau2} shows this graph. 
\end{example}

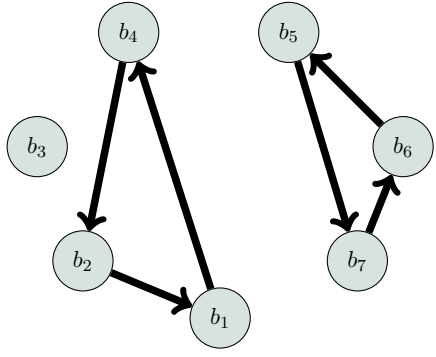
\begin{figure}
\centering
\captionsetup{width=.7\linewidth}
\captionsetup{justification=centering}
\begin{tikzpicture}[x = 4cm, y = 4cm, scale=.38, baseline = (current bounding box.north)]

       \vertex (c1) at (1.6,-1.0)[label,minimum size = 1.0cm,scale=.8,fill=lightgrey] {\textcolor{black}{\large $b_{1}$}};
       \vertex (c2) at (0.4,-0.5)[label,minimum size = 1.0cm,scale=.8,fill=lightgrey] {\textcolor{black}{\large $b_2$}};
       \vertex (c3) at (0,0.5)[label,minimum size = 1.0cm,scale=.8,fill=lightgrey] {\textcolor{black}{\large $b_3$}};
       \vertex (c4) at (0.8,1.5)[label,minimum size = 1.0cm,scale=.8,fill=lightgrey] {\textcolor{black}{\large $b_4$}};
       \vertex (c5) at (2.2,1.5)[label,minimum size = 1.0cm,scale=.8,fill=lightgrey] {\textcolor{black}{\large $b_5$}};
       \vertex (c6) at (3.2,0.5)[label,minimum size = 1.0cm,scale=.8,fill=lightgrey] {\textcolor{black}{\large $b_6$}};
       \vertex (c7) at (2.8,-0.5)[label,minimum size = 1.0cm,scale=.8,fill=lightgrey] {\textcolor{black}{\large $b_7$}};
       
\path
        (c1) edge [style={->, black, line width=1mm}](c4)
        (c4) edge [style={->, black, line width=1mm}](c2) 	
        (c2) edge [style={->, black, line width=1mm}](c1)
        (c5) edge [style={->, black, line width=1mm}](c7)	
        (c7) edge [style={->, black, line width=1mm}](c6)
        (c6) edge [style={->, black, line width=1mm}](c5);
\end{tikzpicture}
\caption{The $\tau$-graph $\mathscr{T}_{\sigma}$ corresponding to $\sigma = (b_5, b_6, b_4, b_2, b_3, b_7, b_1)$.}
\label{fig:mergeandtau2}
\end{figure}

We now define a second graph theoretic tool, the \emph{merge graph}, which will be similar to the $\tau$-graph, but will correspond to a specific permutation rather than to an ordered pair list.

\begin{definition} \label{mergeGraph}
The \textit{merge graph} of $\pi$ is an in-degree at most one, out-degree at most one directed graph $\mathscr{M}_{\pi}=(V, \; E)$, where $V=\{b_1, b_2, \ldots, b_k\}$ and $(b_i, b_j) \in E$ if and only if  $b_i +1 =b_j$ in the permutation $\pi$. 
\end{definition}

Observe that if $\pi$ is a permutation with ordered pair list $\sigma$, then the merge graph $\mathscr{M}_{\pi}$, is an edge subgraph of the $\tau$-graph $\mathscr{T}_{\sigma}$. Moreover, Lemma \ref{lemma:mergecycles} will show that $\mathscr{M}_{\pi}$ is a proper subgraph of $\mathscr{T}_{\sigma}$, and that any acyclic edge subgraph of $T_{\sigma}$ is a merge graph.

\begin{lemma}\label{lemma:mergecycles}
A graph is a merge graph if and only if it is an acyclic edge subgraph of a $\tau$-graph.
\end{lemma}
\begin{proof}

Suppose on the contrary that $\pi$ is a permutation such that its merge graph $\mathscr{M}_{\pi}$ contains an $\ell$-cycle between consecutive vertices $v_1, v_2, \ldots, v_{\ell}$ for some $\ell>0$, and let $b_{x_i}$ be the strategic pile element corresponding to the vertex $v_i$ for $1\le i\le \ell$. Then, by definition of $\mathscr{M}_{\pi}$, $b_{x_1}+1 = b_{x_2} = b_{x_3}-1$, so $b_{x_1} = b_{x_3}-2$. Continuing in this manner, we see that $b_{x_1} = b_{x_{\ell}} - (\ell-1)$. However, as the directed edges among these vertices form a cycle, we also have that $b_{x_1} = b_{x_{\ell}} +1 $. Since $\ell>0$, this is impossible.

Conversely, given an acyclic edge subgraph of a $\tau$-graph, one can construct an ordered pair list with the corresponding merges. This can be done because our only constraint on merges that can occur is cyclic relationships between strategic pile variables.  
\end{proof}

\begin{example}
For $\pi = [5\;4\;6\;3\;2\;1]$ we have 
$\textsf{SP}^*(\pi) = (1,\;3,\;5,\;4) = (b_1,\; b_2,\; b_3,\; b_4)$ and ordered pair list $\sigma =(5,\; 4,\; 3,\; 1) = (b_3,\;b_4,\;b_2,\;b_1)$. Thus, $\psi = (1\; 3\; 5\; 4) = (b_1\; b_2\; b_3 \; b_4)$ and $\sigma^* = (5\; 4\; 3\; 1) = (b_3 \; b_4 \; b_2 \; b_1)$. Note that 
\[
   \tau_{\sigma} = \sigma^* \circ \psi = (b_3 \; b_4 \; b_2 \; b_1)\circ (b_1\; b_2\; b_3 \; b_4) = (b_1)\circ(b_2\; b_4\; b_3).
\]
The cycle $(b_2\; b_4\; b_3)$ in the cycle decomposition of the permutation $\tau_{\sigma}$ indicates that the $\tau$-graph $\mathscr{T}_{\sigma}$ will have edges $(b_2, b_4)$, $(b_4, b_3)$, and $(b_3, b_2)$. Lemma \ref{lemma:mergecycles} implies that any \emph{proper} subset of these three edges can occur in $\mathscr{M}_{\pi}$. Indeed, $b_2+1=b_4$ and $b_4+1=b_3$ in $\pi$, while $b_3+1 \neq b_2$ in $\pi$. The merge graph $\mathscr{M}_{\pi}$ and $\tau$-graph $\mathscr{T}_{\sigma}$ are depicted in Figure \ref{fig:mergeandtau1}.
\end{example}

\begin{figure}
\centering
\captionsetup{width=.7\linewidth}
\captionsetup{justification=centering}
\begin{tikzpicture}[x = 4cm, y = 4cm, scale=.38, baseline = (current bounding box.north)]
       \vertex (c1) at (1.6,-1.0)[label,minimum size = 1.0cm,scale=.8,fill=lightgrey] {\textcolor{black}{\large $b_{2}$}};
       \vertex (c2) at (0,0)[label,minimum size = 1.0cm,scale=.8,fill=lightgrey] {\textcolor{black}{\large $b_1$}};
       \vertex (c3) at (1.6,1.0)[label,minimum size = 1.0cm,scale=.8,fill=lightgrey] {\textcolor{black}{\large $b_3$}};
       \vertex (c4) at (3.1,0.0)[label,minimum size = 1.0cm,scale=.8,fill=lightgrey] {\textcolor{black}{\large $b_4$}};
\path
        (c4) edge [style={->, black, line width=1mm}](c3)
        (c1) edge [style={->, black, line width=1mm}](c4)
        (c3) edge [style={->,dashed, double=black, thick}](c1);
\end{tikzpicture}
\caption{Merge graph of $\pi= [5\;4\;6\;3\;2\;1]$ (solid) as a subgraph of the $\tau$-graph corresponding to $\sigma_{\pi}$ (solid and dashed).}
\label{fig:mergeandtau1}
\end{figure}
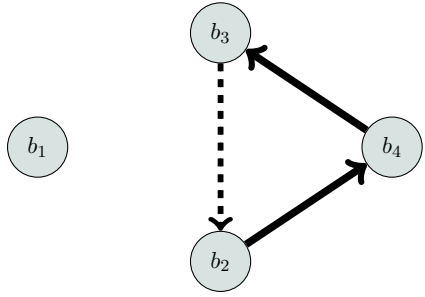

\subsection{Maximum Number of Merges} \label{3.4:mergenumbers}
We will now determine the maximum number of merges that can occur in a permutation in $\textsf{S}_n$ with strategic pile size $k$. This will function to show that $c_{k,i}=0$ for all $i$ greater than a certain value. We start with the following observation. 

\begin{remark} \label{smouchability}
Let $\sigma = (b_{x_1}, b_{x_2}, \ldots,b_{x_{k-1}}, b_1)$. Since $\tau_{\sigma} \in \textsf{S}_k$ is the composition of two cycles of the same length $k$, $\tau_{\sigma}$ is an even permutation. It follows that if $k$ is even, then $\tau_{\sigma}$ cannot be a $k$-cycle. 
\end{remark}

This observation, in conjunction with our graph theoretic tools, will give us our desired result:

\begin{lemma} \label{smouchenumber}
Consider the set $T$ of permutations with strategic piles of size $k$.
\begin{enumerate}
\item{If $k$ is odd, then the number of merges for any permutation in $T$ is at most $k-1$.}
\item{If $k$ is even, then the number of merges for any permutation in $T$ is at most $k-2$.}
\end{enumerate}
\end{lemma}

\begin{proof}
{\flushleft{Case 1: }Let $k$ be odd.} Suppose $\pi \in \textsf{S}_n$ has ordered strategic pile $\textsf{SP}^*(\pi)=(b_1, \ldots, b_k)$. Consider the merge graph $\mathscr{M}_{\pi}$, which will have $k$ vertices. By definition, $\mathscr{M}_{\pi}$ is at most in-degree one and out-degree one. Since Lemma \ref{lemma:mergecycles} gives that $\mathscr{M}_{\pi}$ is acyclic, it follows that the number of edges in $\mathscr{M}_{\pi}$ does not exceed $k-1$. Thus, there are at most $k-1$ merges in $\pi$.  

{\flushleft{Case 2: }Let $k$ be even.} By Remark \ref{smouchability}, the cycle decomposition of $\tau$ does not contain a $k$-cycle, and thus the largest possible cycle in $\mathscr{T}$ is a $(k-1)$-cycle. It follows that there are at most $k-2$ merges in any permutation $\alpha \in T$. 
 \end{proof}

Lemma \ref{smouchenumber} tells us that $c_{k,i} = 0$ for large enough $i$, or in other words, that the number of permutations of $n$ elements with strategic pile of size $k$ can be written
\begin{equation*}
(n-k)!\sum_{i=0}^{t} c_{k,i} \binom{n-(k+1)}{k-i-1}
\end{equation*}
where $t=k-1$ if $k$ is odd and $t=k-2$ if $k$ is even.
 
\subsection{Merge Number Algorithm} \label{subsec:algo}

In the previous subsection, we established that $c_{k,i}=0$ for $i > k-1$ when $k$ is odd and for $i>k-2$ when $k$ is even. We now discuss how to compute the merge numbers corresponding to smaller $i$. We present Algorithm \ref{algorithm} for computing such merge numbers, prove its correctness, and discuss its complexity and what work still needs to be done in order to make this algorithm more efficient.

\begin{alg}[Merge Number Computation]  \label{algorithm}~\\
\emph{Input:} Integers $k$ and $\ell$ where $\ell \leq k-1$ if $k$ is odd and $\ell \leq k-2$ if $k$ is even.\\
\emph{Output:} The merge number $c_{k,\ell}$.
\begin{enumerate}[topsep=0pt]
\item Consider all possible cycle structures for a $\tau$-graph with $k$ vertices. 
\item For each of these cycle structures:
\begin{enumerate}
\item determine the number of ordered pair lists $\sigma = (b_{x_1}, b_{x_2}, \ldots, b_{x_{k-1}}, b_1)$ which yield the given cycle structure.
\item multiply by the number of ways $\ell$ merges can be chosen from the given cycle structure.
\end{enumerate}
\item Sum the results of the calculation for each cycle structure.
\end{enumerate}
\end{alg}

\subsubsection{Correctness and Complexity of Step 1}

Step 1 of Algorithm \ref{algorithm} requires considering all possible cycle structures for a $\tau$-graph with $k$ vertices. Recall that a $\tau$-graph consists only of cycles and isolated vertices. As a result, we can use the following notation to refer to the cycle structure of a $\tau$-graph:

\begin{notation}
Let $[a_1, a_2, \ldots, a_m]$ denote the cycle structure of a $\tau$-graph $\mathscr{T}_{\sigma}$, where $a_1 \geq a_2 \geq \cdots \geq a_m > 0$, and where each $a_i$ corresponds to the number of edges of a cycle in $\mathscr{T}_{\sigma}$. We do not include isolated vertices in our cycle structure representation.
\end{notation}

Since any $\tau$-graph is in-degree at most one and out-degree at most one, a $\tau$-graph on $k$ vertices has at most $k$ edges. As a result, the cycle structure $[a_1, a_2, \ldots, a_m]$ corresponding to $\mathscr{T}_{\sigma}$ satisfies $\sum_{1 \leq i \leq m} a_i \leq k$. Therefore, to consider all possible cycle structures for a $\tau$-graph with $k$ vertices, it would suffice to consider all integer partitions of at most $k$. However, to speed up Step 1, we'd like to be able to consider a smaller set of partitions.

To this end, note that a $\tau$-graph $\mathscr{T}_{\sigma}$ on $k$ vertices contains no self-loops, since this would imply that $b_i+1 = b_i$ for some strategic pile element $b_i$. Moreover, note that the cycle structure of $\mathscr{T}_{\sigma}$ must have an even number of even parts, since $\tau_{\sigma}$ is an even permutation. To summarize,   

\begin{remark} \label{rmk:partition}
Every $\tau$-graph on $k$ vertices has a cycle structure in the form of an integer partition $[a_1, a_2, \ldots, a_m]$, with no parts of size one, with an even number of even parts, and with $\sum_{1 \leq i \leq m} a_i \leq k$.
\end{remark}

\begin{figure}
\centering
\captionsetup{width=.7\linewidth}
\captionsetup{justification=centering}
\begin{tikzpicture}[x = 4cm, y = 4cm, scale=.38, baseline = (current bounding box.north)]

\       \vertex (c1) at (1.3,-1.2)[label,minimum size = 1.0cm,scale=.8,fill=lightgrey] {\textcolor{black}{\large $b_{1}$}};
       \vertex (c2) at (0.2,-0.5)[label,minimum size = 1.0cm,scale=.8,fill=lightgrey] {\textcolor{black}{\large $b_2$}};
       \vertex (c3) at (0.2,0.5)[label,minimum size = 1.0cm,scale=.8,fill=lightgrey] {\textcolor{black}{\large $b_3$}};
       \vertex (c4) at (1.3,1.2)[label,minimum size = 1.0cm,scale=.8,fill=lightgrey] {\textcolor{black}{\large $b_4$}};
       \vertex (c5) at (2.4,0.5)[label,minimum size = 1.0cm,scale=.8,fill=lightgrey] {\textcolor{black}{\large $b_5$}};
       \vertex (c6) at (2.4,-0.5)[label,minimum size = 1.0cm,scale=.8,fill=lightgrey] {\textcolor{black}{\large $b_6$}};

\path
        (c1) edge [style={->, black, line width=1mm}](c6)
        (c6) edge [style={->, black, line width=1mm}](c5) 
        (c5) edge [style={->, black, line width=1mm}](c1)
       	
        (c4) edge [style={->, black, line width=1mm}](c3)
        (c3) edge [style={->, black, line width=1mm}](c2)
        (c2) edge [style={->, black, line width=1mm}](c4);
  
\end{tikzpicture}
\caption{A $\tau$-graph with cycle structure $[3, 3]$}
\label{fig:TauPartition}
\end{figure}
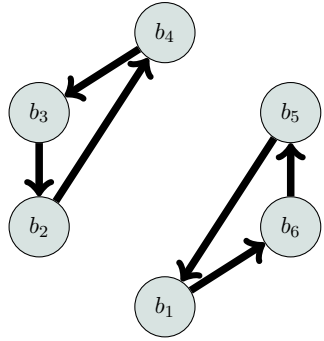

\begin{example}
Consider the graph in Figure \ref{fig:TauPartition}, which has cycle structure $[3, 3]$. Notice that this cycle structure satisfies the conditions described in Remark \ref{rmk:partition}, which are necessary conditions for a cycle structure to correspond to a $\tau$-graph. Indeed, this graph can be derived from the ordered pair list $\sigma = (b_5, b_3, b_4, b_2, b_6, b_1)$, and is thus a $\tau$-graph. 
\end{example}

We will not prove that every integer partition satisfying the conditions of Remark \ref{rmk:partition} corresponds to the cycle structure of a $\tau$-graph on $k$ vertices (i.e. the converse of Remark \ref{rmk:partition}), since it will not affect the correctness of our algorithm. If it happens that we consider in Step 1 a partition that does not correspond to the cycle structure of a $\tau$-graph on $k$ vertices, Step 2(a) will yield a zero, so we will not be over-counting.  

Unfortunately, the best way currently known to determine the set of integer partitions satisfying the properties of Remark \ref{rmk:partition} is the brute force method of checking every partition of every integer from $1$ to $k$. Since the number of integer partitions of an integer $n$ grows exponentially with $n$ \cite{ramanujan}, Step 1 is inefficient for large $k$. 
 
\subsubsection{Correctness and Complexity of Step 2(a)} \label{subsubsec:2a}

Step 2(a) requires determining the number of ordered pair lists $\sigma = (b_{x_1}, b_{x_2}, \ldots, b_{x_{k-1}}, b_1)$ which yield a given $\tau$-graph cycle structure. As in Step 1, this can be done through brute force; namely, one can generate all $O(k!)$ possible ordered pair lists, and for each ordered pair list $\sigma$, can compute $\tau_{\sigma} = \sigma^* \circ \psi$ to determine whether $\tau_{\sigma}$ has the given cycle structure. Since each computation of $\tau_{\sigma}$ requires $O(k)$ time, Step 2(a) can be completed in $O(k \cdot k!)$ with this brute force method.    

It is possible, however, that this step could be accomplished in polynomial time using a recursive formula. We will now derive such a formula, though a method for efficiently computing the base cases for this formula is currently unknown. Our derivation will involve understanding the relationship between $\tau$-graphs on $k$ vertices with cycle structure $[a_1, a_2, \ldots, a_m]$ and $\tau$-graphs on $k-1$ vertices with the same cycle structure. To build intuition for this relationship, let us consider an example.

\begin{example} \label{ex:Recursion}
The aforementioned relationship will be established by rotating and removing vertices from $\tau$-graphs. For example, consider the $\tau$-graph corresponding to the ordered pair list $(b_2,b_5,b_4,b_3,b_1)$ (see Figure \ref{fig:recursiveUnrotated}). In Lemma \ref{groupaction} we will prove that any rotation of a $\tau$-graph is also a $\tau$-graph. In particular, any rotation of the $\tau$-graph in Figure \ref{fig:recursiveUnrotated} is a $\tau$-graph; Figure \ref{fig:recursiveRotated} shows the rotation that is the $\tau$-graph corresponding to the ordered pair list $(b_5,b_3,b_4,b_2,b_1)$. Note that this $\tau$-graph has $b_5$ as an isolated vertex.  

In a graph with $b_k$ (in this case $b_5$) as an isolated vertex, removing $b_k$ will give us another $\tau$-graph; this is because when $b_k$ is an isolated vertex, removing the vertex $b_k$ corresponds to removing $b_k+1 \ldots b_k$ from the beginning of the pair ordering. This will leave $b_{k-1}+1$ at the beginning of the pair ordering, which will yield a valid ordered pair list on $k-1$ strategic pile elements. Figure \ref{fig:recursiveRemoved} shows the $\tau$-graph corresponding to the ordered pair list $(b_3,b_4,b_2,b_1)$ that occurs when $b_5$ is removed from our example $\tau$-graph.
\end{example}

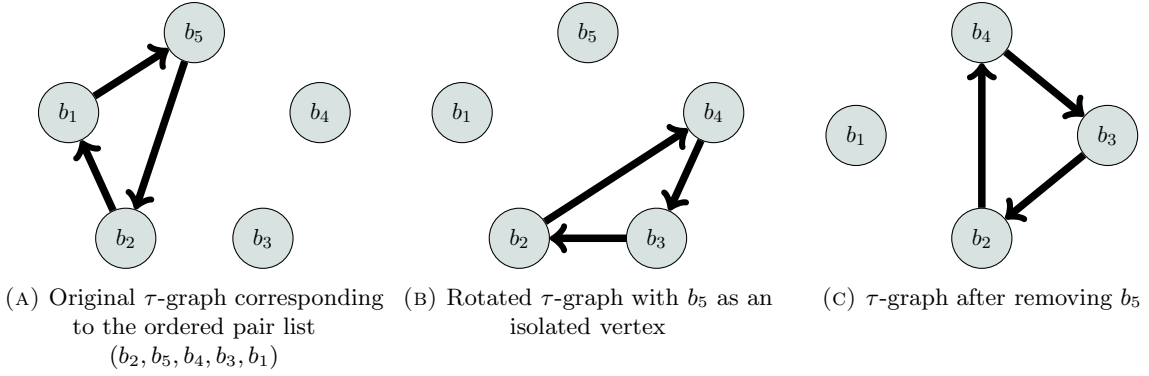
\begin{figure}[t]
\centering
\begin{subfigure}[t]{0.31\textwidth}
\centering
\begin{tikzpicture}[x = 4cm, y = 4cm, scale=.38, baseline = (current bounding box.north)]

       \vertex (c1) at (0.2,0.5)[label,minimum size = 1.0cm,scale=.8,fill=lightgrey] {\textcolor{black}{\large $b_1$}};
       \vertex (c2) at (0.7,-0.6)[label,minimum size = 1.0cm,scale=.8,fill=lightgrey] {\textcolor{black}{\large $b_2$}};
       \vertex (c3) at (1.9,-0.6)[label,minimum size = 1.0cm,scale=.8,fill=lightgrey] {\textcolor{black}{\large $b_3$}};
       \vertex (c4) at (2.4,0.5)[label,minimum size = 1.0cm,scale=.8,fill=lightgrey] {\textcolor{black}{\large $b_4$}};
       \vertex (c5) at (1.3,1.2)[label,minimum size = 1.0cm,scale=.8,fill=lightgrey] {\textcolor{black}{\large $b_5$}};
       
\path
        (c2) edge [style={->, black, line width=1mm}](c1)
        (c5) edge [style={->, black, line width=1mm}](c2) 
        (c1) edge [style={->, black, line width=1mm}](c5);
\end{tikzpicture}
\caption{\centering Original $\tau$-graph corresponding to the ordered pair list $(b_2,b_5,b_4,b_3,b_1)$} \label{fig:recursiveUnrotated}
\end{subfigure}
\begin{subfigure}[t]{0.31\textwidth}
\centering
\begin{tikzpicture}[x = 4cm, y = 4cm, scale=.38, baseline = (current bounding box.north)]

       \vertex (c1) at (0.2,0.5)[label,minimum size = 1.0cm,scale=.8,fill=lightgrey] {\textcolor{black}{\large $b_1$}};
       \vertex (c2) at (0.7,-0.6)[label,minimum size = 1.0cm,scale=.8,fill=lightgrey] {\textcolor{black}{\large $b_2$}};
       \vertex (c3) at (1.9,-0.6)[label,minimum size = 1.0cm,scale=.8,fill=lightgrey] {\textcolor{black}{\large $b_3$}};
       \vertex (c4) at (2.4,0.5)[label,minimum size = 1.0cm,scale=.8,fill=lightgrey] {\textcolor{black}{\large $b_4$}};
       \vertex (c5) at (1.3,1.2)[label,minimum size = 1.0cm,scale=.8,fill=lightgrey] {\textcolor{black}{\large $b_5$}};
       
\path
        (c4) edge [style={->, black, line width=1mm}](c3)
        (c3) edge [style={->, black, line width=1mm}](c2) 
        (c2) edge [style={->, black, line width=1mm}](c4);
  
\end{tikzpicture}
\caption{\centering Rotated $\tau$-graph with $b_5$ as an isolated vertex} \label{fig:recursiveRotated}
 \end{subfigure} 
 \begin{subfigure}[t]{0.31\textwidth}
\centering
\begin{tikzpicture}[x = 4cm, y = 4cm, scale=.38, baseline = (current bounding box.north)]

       \vertex (c1) at (0.2,0.3)[label,minimum size = 1.0cm,scale=.8,fill=lightgrey] {\textcolor{black}{\large $b_1$}};
       \vertex (c2) at (1.3,-0.6)[label,minimum size = 1.0cm,scale=.8,fill=lightgrey] {\textcolor{black}{\large $b_2$}};
       \vertex (c3) at (2.4,0.3)[label,minimum size = 1.0cm,scale=.8,fill=lightgrey] {\textcolor{black}{\large $b_3$}};
       \vertex (c4) at (1.3,1.2)[label,minimum size = 1.0cm,scale=.8,fill=lightgrey] {\textcolor{black}{\large $b_4$}};
       
\path
        (c2) edge [style={->, black, line width=1mm}](c4)
        (c4) edge [style={->, black, line width=1mm}](c3) 
        (c3) edge [style={->, black, line width=1mm}](c2);

\end{tikzpicture}
\caption{\centering $\tau$-graph after removing $b_5$} \label{fig:recursiveRemoved}
 \end{subfigure}
 
 \caption{$\tau$-graphs for Example \ref{ex:Recursion}} \label{fig:recursionEx}
\end{figure}

We will now formalize this idea. Let $X_{k,[a_1, a_2, \ldots, a_m]}$ be the set of $\tau$-graphs with $k$ vertices and with cycle structure $[a_1, a_2, \ldots, a_m]$. We are interested in finding a relationship between $\vert X_{k,[a_1, a_2, \ldots, a_m]} \vert$ and $\vert X_{k-1,[a_1, a_2,  \ldots, a_m]} \vert$. We begin by showing that if a given $\tau$-graph is in $X_{k,[a_1, a_2, \ldots,  a_m]}$, then so are all rotations of that graph. Thus we get a group action on $X_{k,[a_1, a_2, \ldots, a_m]}$.

\begin{lemma} \label{groupaction}
Let $\sigma = (b_{x_1}, \ldots, b_{x_{k-1}}, b_1)$ be an ordered pair list, and define $\varphi : \mathbb{Z}_k \times X_{k,[a_1, \ldots, a_m]} \rightarrow X_{k,[a_1, \ldots, a_m]}$ as
$$\varphi(i,\tau_{\sigma}) = \beta^i \circ \tau_{\sigma} \circ \beta^{-i},$$ 
where $\beta = (1 \; 2 \; 3 \; \cdots \; k)$. Then $\varphi$ is a group action. 
\end{lemma}

\begin{proof}
Let $\sigma = (b_{x_1}, \ldots, b_{x_{k-1}}, b_1)$ be an ordered pair list and suppose $\tau_{\sigma}$ has cycle structure $[a_1, \ldots, a_m]$.

We first show that $\varphi(i,\tau_{\sigma}) \in X_{k,[a_1, \ldots, a_m]}$. Since conjugation preserves cycle structure, it is clear that $\varphi(i,\tau_{\sigma})$ will be a graph on $k$ vertices with cycle structure $[a_1, \ldots, a_m]$. We have left to show that $\varphi(i,\tau_{\sigma})$ corresponds to an ordered pair list $\sigma'$ (i.e. that $\varphi(i, \tau_{\sigma}) = \tau_{\sigma'}$).

Recall that $\tau_{\sigma} = \sigma^* \circ \psi$. Therefore,
\begin{align*}
\varphi(i,\tau_{\sigma}) = \beta^i \circ \tau_{\sigma} \circ \beta^{-i} = \beta^i \circ \sigma^* \circ \psi \circ \beta^{-i} = \beta^i \circ \sigma^* \circ \beta^{-i} \circ \beta^i \circ \psi \circ \beta^{-i}.
\end{align*}
Observe that $\beta^i \circ \psi \circ \beta^{-i} = \psi$. Furthermore, $\beta^i \circ \sigma^* \circ \beta^{-i} = \sigma'$ is a $k$-cycle containing the elements $\{ b_1, \ldots, b_k \}$, and thus represents an ordered pair list. As a result, $\varphi(i,\tau_{\sigma}) = \psi \circ \sigma' = \tau_{\sigma'}$. It follows that $\varphi(i,\tau_{\sigma}) \in X_{k,[a_1, \ldots, a_m]}$, as desired. 

Finally, we check that $\varphi$ satisfies the axioms of group actions. Clearly, $\beta^0 = \beta^k$ is the identity permutation, and therefore $\varphi(0, \tau_{\sigma}) = \tau_{\sigma}$. In addition, 
\begin{equation*}
\varphi(i+j, \tau_{\sigma}) = \beta^{i+j} \circ \tau_{\sigma} \circ \beta^{-(i+j)} = \beta^i \beta^{j} \circ \tau_{\sigma} \circ \beta^{-j} \beta^{-i} = \varphi(i, \varphi(j, \tau_{\sigma})).
\end{equation*}
\end{proof}

We can use Lemma \ref{groupaction} to give a process for deriving $X_{k,[a_1, a_2, \ldots, a_m]}$ from $X_{k-1,[a_1, a_2, \ldots, a_m]}$.

\begin{lemma} \label{rotationequiv}
Let $G = (V, E)$ with $V = \{b_1,\ldots,b_k\}$, and let $a_1, \ldots, a_m \in \mathbb{Z}$ be such that $a_1 + a_2 +\ldots + a_m < k$. Then $G \in X_{k,[a_1, a_2, \ldots, a_m]}$ if and only if there exists some $G_r = (V_r, E_r) \in \orb_{\mathbb{Z}_k}(G)$ such that $b_k$ is an isolated vertex in $G_r$ and $G_s := (V_r\setminus \{ b_k \}, E_r) \in X_{k-1,[a_1, a_2, \ldots, a_m]}$.
\end{lemma}

\begin{proof}
Assume $G \in X_{k,[a_1, a_2, \ldots, a_m]}$. Since $a_1 + a_2 + \cdots + a_m < k$, there exists at least one isolated vertex in $G$. Therefore, some rotation of $G$ has $b_k$ as an isolated vertex. Let this rotation be $G_r$ and let $G_s$ be defined as in the lemma statement.
We have left to show that $G_s$ is a $\tau$-graph. Since $G_r$ has $b_k$ as an isolated vertex, any permutation with ordered pair list corresponding to $G_r$ must be of the form
$$[b_k+1 \; \cdots \; b_k \; b_{k-1}+1 \; \cdots \; b_{\ell_1} \; b_{\ell_2-1}+1 \; \cdots \; b_{\ell_{k-1}} \; b_{\ell_{k-2}-1}+1 \; \cdots \; b_1].$$
Removing the $b_k \; b_{k-1}+1$ pair, we are left with an ordered pair list $\sigma = (b_{\ell_1}, b_{\ell_2}, \ldots,  b_{\ell_{k-2}}, b_1)$ of $k-1$ strategic pile elements. This ordered pair list clearly corresponds to the graph $G_s$, meaning $G_s$ is a $\tau$-graph; it follows that $G_s \in X_{k-1,[a_1, a_2, \ldots, a_m]}$.

Conversely, Let $G_s = (V_s, E_s)$ be an arbitrary graph in $X_{k-1,[a_1, a_2, \ldots, a_m]}$, and let $G_r := (V_s \cup \{b_k\}, E_s)$. Then $G_r$ is also a $\tau$-graph, since the ordered pair list associated with $G_r$ is the ordered pair list associated with $G_s$ with the addition of $b_k$ as the first element. Note that $\orb_{\mathbb{Z}_k}(G_r)$ is a subset of $X_{k,[a_1,a_2,\ldots,a_m]}$ since $\mathbb{Z}_k$ acts on $X_{k,[a_1,a_2,\ldots ,a_m]}$. It follows that $G \in X_{k,[a_1,\ldots ,a_m]}$ for any $G \in \orb_{\mathbb{Z}_k}(G_r)$. 
\end{proof}

Using Lemma \ref{rotationequiv} we can determine the number of elements in $X_{k,[a_1,a_2,\ldots,a_m]}$ by adding a vertex to each graph in $X_{k-1,[a_1,a_2,\ldots ,a_m]}$, and then considering all rotations of each of those graphs. However, the graphs formed through this method are not necessarily distinct. One of the reasons this is true is due to the fact that two graphs in $X_{k-1,[a_1,a_2,\ldots ,a_m]}$ may be in the same orbit when the vertex $b_k$ is added. The following lemma addresses this issue.

\begin{lemma} \label{edgeoverstab}
For all $\tau \in X_{k,[a_1,a_2, \ldots ,a_m]}$, let $Z_{\tau}$ be the set of all $\tau$-graphs in the orbit of $\tau$ under $\mathbb{Z}_k$ which do not have $b_k$ as an isolated vertex. Then 
$$|Z_{\tau}| = \frac{a_1+a_2+ \cdots + a_m}{|\stab(\tau)|}.$$
\end{lemma}

\begin{proof}
Define $\ell:= a_1+a_2+ \cdots + a_m$, and note that this is the number of edges in any member of $X_{k,[a_1,a_2,\ldots ,a_m]}$. Let $\tau \in X_{k,[a_1,a_2, \ldots,a_m]}$ and label the edges of $\tau$ as $e_1,e_2,\ldots,e_{\ell}$. Let $z_i$ be the rotation of $\tau$ such that $e_i$ is a directed edge terminating at $b_k$. Then $Z_{\tau} = \{z_1,z_2,...,z_{\ell} \}$, since $b_k$ is not an isolated vertex if and only if some edge points to $b_k$. However these $z_i$ are not necessarily distinct.

Observe that for a given $i$, the number if times $z_i$ appears in $Z_{\tau}$ is given by $\vert \stab(z_i) \vert $. Moreover, since $z_i \in \orb_{\mathbb{Z}_k}(\tau)$, we have that $\vert \stab(z_i) \vert = \vert \stab(\tau) \vert$. Therefore, 
$$|Z_{\tau}|= \frac{\ell}{|\stab(\tau)|}  = \frac{a_1+a_2+ \cdots +a_m}{|\stab(\tau)|}.$$

\end{proof}

We now have what we need to prove the main relationship between $ |X_{k,[a_1,a_2, \ldots ,a_m]}|$ and $|X_{k-1,[a_1,a_2, \ldots,a_m]}|$.

\begin{thm} \label{rec_struc}
For any $a_1, a_2, \ldots, a_m \in \mathbb{Z}$ such that $a_1 + a_2 + \ldots + a_m < k$,
$$ |X_{k,[a_1,a_2, \ldots ,a_m]}| = \frac{k|X_{k-1,[a_1,a_2, \ldots ,a_m]}|}{k-(a_1+a_2+\cdots +a_m)}$$
\end{thm}

\begin{proof}
By Lemma \ref{rotationequiv}, in order to count $|X_{k,[a_1, a_2, \ldots ,a_m]}|$, we can add a vertex to every graph in $X_{k-1,[a_1,a_2, \ldots ,a_m]}$ and consider all rotations of these new graphs. Each of these graphs has $k$ possible rotations. However, this does not produce distinct elements of $X_{k,[a_1,a_2, \ldots ,a_m]}$. In fact, for each $\tau \in  X_{k,[a_1,a_2,\ldots ,a_m]}$, we have counted it $|\stab_{\mathbb{Z}_k}(\tau)|\cdot|\orb_{\mathbb{Z}_k}(\tau)\setminus Z_{\tau}|$ times. Since the addition of the vertex $b_k$ can cause non-isomorphic graphs in $X_{k-1,[a_1,a_2, \ldots ,a_m]}$ to be in the same orbit under $\mathbb{Z}_k$ (see Figure \ref{Fig:OrbitGraphs}), we have over-counted each orbit in $X_{k,[a_1,a_2, \ldots ,a_m]}$ by a factor of $|\orb_{\mathbb{Z}_k}(\tau)\setminus Z_{\tau}|$. Due to rotational symmetry, we over-count $\tau \in \orb_{\mathbb{Z}_k}(\tau)$ by a factor of $|\stab_{\mathbb{Z}_k}(\tau)|$.

By Lemma \ref{edgeoverstab}, for all $\tau \in X_{k,[a_1,a_2, \ldots, a_m]}$, we have that $\displaystyle |Z_{\tau}|= \frac{a_1+a_2+ \cdots + a_m}{|\stab(\tau)|}$. Recall that $Z_{\tau} \subseteq \orb_{\mathbb{Z}_k}(\tau)$. Therefore,
\begin{align*}
|\stab_{\mathbb{Z}_k}(\tau)|\cdot |\orb_{\mathbb{Z}_k}(\tau)\setminus Z_{\tau}| &= |\stab_{\mathbb{Z}_k}(\tau)| (|\orb_{\mathbb{Z}_k}(\tau)|-|Z_{\tau}|)\\
&= |\stab_{\mathbb{Z}_k}(\tau)| \cdot |\orb_{\mathbb{Z}_k}(\tau)| - |\stab_{\mathbb{Z}_k}(\tau)|\cdot |Z_{\tau}| \\
&= |\stab_{\mathbb{Z}_k}(\tau)|\cdot|\orb_{\mathbb{Z}_k}(\tau)|-(a_1+a_2 \cdots +a_m)
\end{align*}

Then, by the orbit stabilizer theorem, $|\stab_{\mathbb{Z}_k}(\tau)| \cdot |\orb_{\mathbb{Z}_k}(\tau)| = |\mathbb{Z}_k|=k$, so we have counted each rotation $k-(a_1+a_2+\cdots+a_m)$ times. Therefore, 
$$|X_{k,[a_1,a_2, \ldots ,a_m]}| = \frac{k |X_{k-1,[a_1,a_2,\ldots,a_m]}|}{k-(a_1+a_2+\cdots +a_m)}.$$ 

\end{proof}

\begin{figure}
\centering
\captionsetup{width=.7\linewidth}
\captionsetup{justification=centering}
\begin{tikzpicture}[x = 4cm, y = 4cm, scale=.38, baseline = (current bounding box.north)]

       \vertex (c4) at (1.6,-1.0)[label,minimum size = 1.0cm,scale=.8,fill=lightgrey] {\textcolor{black}{\large $b_{4}$}};
       \vertex (c5) at (0.5,-0.5)[label,minimum size = 1.0cm,scale=.8,fill=lightgrey] {\textcolor{black}{\large $b_5$}};
       \vertex (c6) at (-.1,0.5)[label,minimum size = 1.0cm,scale=.8,fill=lightgrey] {\textcolor{black}{\large $b_6$}};
       \vertex (c7) at (0.5,1.5)[label,minimum size = 1.0cm,scale=.8,fill=lightgrey] {\textcolor{black}{\large $b_7$}};
       \vertex (c8) at (1.6,2.0)[label,minimum size = 1.0cm, scale=.8,fill=lightgrey] {\textcolor{black}{\large $\star$}};
       \vertex (c1) at (2.7,1.5)[label,minimum size = 1.0cm,scale=.8,fill=lightgrey] {\textcolor{black}{\large $b_1$}};
       \vertex (c2) at (3.3,0.5)[label,minimum size = 1.0cm,scale=.8,fill=lightgrey] {\textcolor{black}{\large $b_2$}};
       \vertex (c3) at (2.7,-0.5)[label,minimum size = 1.0cm,scale=.8,fill=lightgrey] {\textcolor{black}{\large $b_3$}};
       
\path
        (c2) edge [style={->, black, line width=1mm}](c1)
        (c3) edge [style={->, black, line width=1mm}](c2) 	
        (c1) edge [style={->, black, line width=1mm}](c3)
        (c5) edge [style={->, black, line width=1mm}](c4)	
        (c7) edge [style={->, black, line width=1mm}](c5)
        (c4) edge [style={->, black, line width=1mm}](c7);
\end{tikzpicture}
\hspace{1cm}
\begin{tikzpicture}[x = 4cm, y = 4cm, scale=.38, baseline = (current bounding box.north)]

       \vertex (c4) at (1.6,-1.0)[label,minimum size = 1.0cm,scale=.8,fill=lightgrey] {\textcolor{black}{\large $b_{6}$}};
       \vertex (c5) at (0.5,-0.5)[label,minimum size = 1.0cm,scale=.8,fill=lightgrey] {\textcolor{black}{\large $b_7$}};
       \vertex (c6) at (-.1,0.5)[label,minimum size = 1.0cm,scale=.8,fill=lightgrey] {\textcolor{black}{\large $\star$}};
       \vertex (c7) at (0.5,1.5)[label,minimum size = 1.0cm,scale=.8,fill=lightgrey] {\textcolor{black}{\large $b_1$}};
       \vertex (c8) at (1.6,2.0)[label,minimum size = 1.0cm, scale=.8,fill=lightgrey] {\textcolor{black}{\large $b_2$}};
       \vertex (c1) at (2.7,1.5)[label,minimum size = 1.0cm,scale=.8,fill=lightgrey] {\textcolor{black}{\large $b_3$}};
       \vertex (c2) at (3.3,0.5)[label,minimum size = 1.0cm,scale=.8,fill=lightgrey] {\textcolor{black}{\large $b_4$}};
       \vertex (c3) at (2.7,-0.5)[label,minimum size = 1.0cm,scale=.8,fill=lightgrey] {\textcolor{black}{\large $b_5$}};
       
\path
        (c2) edge [style={->, black, line width=1mm}](c1)
        (c3) edge [style={->, black, line width=1mm}](c2) 	
        (c1) edge [style={->, black, line width=1mm}](c3)
        (c5) edge [style={->, black, line width=1mm}](c4)	
        (c7) edge [style={->, black, line width=1mm}](c5)
        (c4) edge [style={->, black, line width=1mm}](c7);
\end{tikzpicture}
\caption{Two non-isomorphic graphs that will be in the same $\mathbb{Z}_k$-orbit after the addition of the vertex $\star$.}
\label{Fig:OrbitGraphs}
\end{figure}
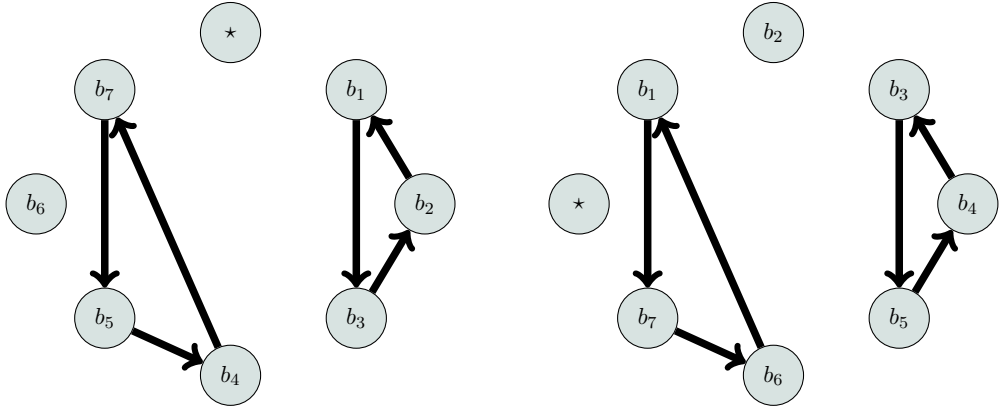

This recursive relationship could be useful for addressing Step 2(a) of the merge number algorithm. However it is only useful when the base cases, $|X_{\ell,[a_1,a_2,\ldots ,a_m]}|$ (where $\ell = a_1 + a_2 + \cdots + a_m$), are already known. Unfortunately, there is no known efficient way to compute these base cases. Using brute force in the same way as we can for Step 2(a) (see the beginning of Sub-subsection \ref{subsubsec:2a}), these base cases could be computed in $O(\ell \cdot \ell!)$ time. Since $\ell = O(k)$ in the worst case, this is not a significant improvement over the original brute force algorithm for Step 2(a).

\subsubsection{Step 2(b)}

Step 2(b) of the algorithm requires determining the number of ways $\ell$ merges can be picked with the given cycle structure. From Lemma \ref{lemma:mergecycles}, we can know that this is equivalent to choosing $\ell$ edges so that no cycle is formed.

Given a graph with cycle structure $[a_1, a_2, \ldots, a_m]$, let $e$ be the number of edges in the graph. This means that $e = a_1 + a_2 + \cdots + a_m$. The total number of ways to chose $\ell$ edges is $\binom{e}{\ell}$. The total number of ways to chose $\ell$ edges that include at least one cycle can be found using the inclusion-exclusion principle as shown below:
$$\sum \limits_{i=1}^{m} (-1)^{i+1} \sum \Bigg\{\binom{e-a_{k_1}-a_{k_2}- \cdots - a_{k_i}}{\ell-a_{k_1}-a_{k_2}- \cdots - a_{k_i}}: 1 \leq k_1,\ldots,k_i \leq m, \text{ all distinct} \Bigg\}$$
Subtracting this from $\binom{e}{\ell}$ yields the total number of ways to to chose $\ell$ edges without picking a cycle, which is what we wanted.

\section{Future work}

According to Theorem \ref{thm:maxstrpileodd}, for an odd natural number $n$, the number of elements of $\textsf{S}_n$ that have a maximum size strategic pile is $2\cdot(n-2)!$ This number is related to the number of factorizations given in the following result from Bertram and Wei:
\begin{thm}[\cite{BertramWei}]\label{thm:BW}
For $n\ge 3$, each odd permutation in $\textsf{S}_n$ has exactly $2(n-2)!$ factorizations of the form $\alpha\circ\beta$ where $\alpha$ is an $n$-cycle and $\beta$ is an $(n-1)$-cycle. 
\end{thm}

Viewing Theorem \ref{thm:BW} in our context, let $n\geq 3$ be an odd integer, and let $\pi$ be an element of $\textsf{S}_n$. With $X_n$, $Y_{\pi}$ and $C_{\pi}$ as defined in  equations (\ref{eq:X_ndef}), (\ref{eq:Y_pidef}) and (\ref {eq:C_pidef}), we are considering factorizations of $X_n$ of the form
\begin{equation*} 
 X_n = Y_{\pi}^{-1} \circ  C_{\pi},
\end{equation*}
where $C_{\pi}$ is a single cycle of length $n$, while $X_n$ and $Y_{\pi}$ are cycles of length $n+1$. Applying Theorem \ref{thm:BW}, we see that according that theorem there are $2(n-1)!$ factorizations of $X_n$ of the form $\mu \circ \nu$ where $\mu$ is an $(n+1)$-cycle and $\nu$ is an $n$-cycle. In each of these cases, we can write $\mu$ as a $Y_{\pi}^{-1}$ for some $\pi\in\textsf{S}_n$, and for $2(n-2)!$ of these $\pi$ the corresponding $\nu$ is a $C_{\pi}$ of the form $(0\; n\; i\; \cdots)$. 

\begin{example}\label{ex:cycles} Consider $n=5$. The following table indicates that $X_5$ has factorizations into a $6$-cycle and a $5$-cycle for which the corresponding permutations $\pi$ have various strategic pile sizes.

\begin{center}
\begin{tabular}{|c|c|c|c|}\hline
$\pi$ & $C_{\pi}$ & Strategic pile & Strategic pile size \\ \hline
$\lbrack 2\; 4\; 1\; 3\; 5\rbrack$ & $(0\; 4\; 3\; 2\; 1)$ & $\emptyset$ & $0$ \\ \hline
$\lbrack 5\; 2\; 3\; 1\; 4\rbrack$ & $(0\; 3\; 1\; 5\; 4)$ & $\{4\}$ & $1$ \\ \hline
$\lbrack 2\; 1\; 5\; 3\; 4\rbrack$ & $(0\; 2\; 5\; 4\; 1\}$ & $\{1,\;4\}$ & $2$ \\ \hline
$\lbrack 3\; 5\; 1\; 2\; 4\rbrack$ & $(0\; 5\; 4\; 3\; 2)$ & $\{2,\;3,\;4\}$ & $3$ \\ \hline
\end{tabular}\\
\end{center}
\end{example}

Thus, it can happen that the cycle $C_{\pi}$ of length $n$ in the factorization of $X_n$ represents a strategic pile of  size less than the maximal possible size for $n$. It would be interesting to determine, for odd integers $n$ and for each strategic pile size $0\le k\le n-2$ how many of the permutations in $\textsf{S}_n$ for which $C_{\pi}$ is a cycle of length $n$ have strategic pile size $k$. We have also not addressed the analogous question for the case when $n$ is an even integer.

In addition to the problem just described, we would like to either (1) improve the merge number algorithm described in Subsection \ref{subsec:algo} or (2) construct an alternative algorithm for computing merge numbers. 

Accomplishing (1) would require improving the following aspects of our algorithm. Let $k$ indicate strategic pile size. Recall that Step 1 of this algorithm requires determining the set of integer partitions of $k$ with no parts of size one, and with an even number of even parts. As previously mentioned, the number of integer partitions of $k$ grows exponentially in $k$ \cite{ramanujan}, meaning Step 1 is inefficient for large $k$. To make this step of the algorithm less costly, we would like a better method for computing the number of partitions with the aforementioned properties. Recall also that Step 2(a) of this algorithm can be done through brute force in $O(k \cdot k!)$ time. We offer a recursive method for completing Step 2(a) with runtime polynomial in $k$. However, this recursive method is only useful when the base cases, $\vert X_{\ell,[a_1,a_2,\ldots,a_m]} \vert$ (where $\ell = a_1+a_2+\cdots +a_m$), are already known. Unfortunately, the best known method for computing the base cases of this algorithm requires $O(\ell \cdot \ell!)$ time. Since this is no better than the brute force method for Step 2(a), we would like an efficient method for computing base cases so that our recursive method can be used to make Step 2(a) more efficient. 

Alternatively, it would be ideal to (2) construct an algorithm for computing merge numbers that completely circumvents the dependency on exponential time computations. However, due to the nature of merge numbers described above, it seems that these dependencies might be unavoidable. Consequently, it is not clear how realistic it would be to accomplish (2).

\section*{Acknowledgments}

This work was supported by NSF grant DMS-1359425 as a part of the 2016 Boise State University Complexity Across Disciplines Research Experience for Undergraduates (REU).

\end{document}